\documentclass[12pt]{article}
\usepackage{amsfonts,amssymb}
\usepackage{graphicx} 
\usepackage{float}
\usepackage{color}
\usepackage{amsthm, amssymb, latexsym, amsmath, bm 
}
\usepackage[margin=1.35 in]{geometry}
\usepackage{enumerate}

\usepackage{dsfont}

\usepackage{hyperref}
\usepackage{cleveref}
\usepackage[dvipsnames]{xcolor}

\usepackage[font=small,labelfont=bf]{caption}
\usepackage{subcaption}

\hypersetup{
	colorlinks   = true, 
	urlcolor     = blue, 
	linkcolor    = blue, 
	citecolor   = red 
}

\newtheorem{theorem}{Theorem}[section]

\newtheorem{lemma}[theorem]{Lemma}
\newtheorem{proposition}[theorem]{Proposition}

\theoremstyle{definition}

\numberwithin{theorem}{section}
\numberwithin{equation}{section}

\newcommand{\Del}{\Delta}
\newcommand{\eps}{\varepsilon}

\newcommand{\R}{{\mathbb R}}
\newcommand{\Z}{{\mathbb Z}}
\newcommand{\bC}{{\mathbb C}}

\newcommand{\cB}{\mathcal{B}}
\newcommand{\cF}{\mathcal{F}}

\newcommand{\cH}{\mathcal{H}}
\newcommand{\cM}{\mathcal{M}}         
\newcommand{\cS}{\mathcal{S}}
\newcommand{\cU}{\mathcal{U}}
\newcommand{\cW}{\mathcal{W}}

\newcommand{\Ran}[1]{\operatorname{Ran}\, #1}
\newcommand{\spec}{\mathrm {spec}}

\renewcommand{\Re}{\operatorname{Re}}
\renewcommand{\Im}{\operatorname{Im}}

\newcommand{\dist}{{\rm{dist\, }}}

\newcommand{\p}{\partial}

\newcommand{\DETAILS}[1]{}

\newcommand{\variants}[1]{}

\newcommand{\N}{\mathbb{N}_0}

\begin{document}
	\title{Near-pulse solutions of the FitzHugh-Nagumo 
		equations on cylindrical surfaces}
	\author{A. Talidou, A. Burchard, and I. M. Sigal}

	\maketitle

\begin{abstract} We introduce a geometrical extension of the  FitzHugh-Nagumo 
equations describing propagation of electrical impulses in nerve axons.
In this extension, the axon is modelled as a warped cylinder, rather than a straight line, as  is usually done, while pulses propagate on its surface, as is the case with real axons.

We prove the stability of electrical impulses for a standard cylinder and existence and stability of pulse-like solutions for  warped cylinders whose radii 
are small  and vary slowly along their lengths. 
\end{abstract}


\section{Introduction}\label{sec:introduction}


The FitzHugh-Nagumo system (\cite{FH61, Nagumo62}), 
modelling the propagation of electric impulses in nerve axons, is a simplified version of the Hodgkin-Huxley system~\cite{HH52} and is given as
\begin{equation}\label{FHN}
\begin{split}
\p_t u_1 &= \p_{x}^2 u_1 + f(u_1) - u_2,\\  
\p_t u_2 &= \eps(u_1-\gamma u_2)\,,
\end{split}
\end{equation}
where $u_1$ and $u_2$ are real functions of $x\in \R$ and $t\ge 0$, the parameters $\eps$ and $\gamma$ are chosen to be
positive and small and $f$ (the reaction term) is given by the cubic polynomial 
\begin{equation*}
f(u_1) := -u_1(u_1 - \alpha)(u_1 - 1),
\end{equation*} 
for $0<\alpha<\frac12$.  Here, an axon is modelled by a straight line 
without an internal geometric structure. 

In our work, we make a first step in taking into account 
the geometry of the axon, 
 namely, a cylindrical cable-like fiber, with electrical 
signals propagating on its surface. Thus we consider an extension of the FitzHugh-Nagumo (FHN) system on a cylindrical surface, $\cS$. The system 
has the form
\begin{equation}\label{eq:FHN}
\begin{split}
\p_t u_1 &= \Del_{\cS} u_1 + f(u_1) - u_2\,,\\  
\p_t u_2 &= \eps(u_1-\gamma u_2)\,,
\end{split}
\end{equation}
where $\Del_{\cS}$ denotes the 
Laplace-Beltrami operator on $\cS$ and  $\eps$, $\gamma$ and $f$  are the same as above. As in the original FitzHugh-Nagumo system, $u_1$ is the electrical potential across
the axon membrane (fast variable), and $u_2$ is the $Na^+$ channel activation and
inactivation parameters lumped into a single unknown (slow, recovery variable).

We call  Eq.~\eqref{eq:FHN} the {\it cylindrical  FitzHugh-Nagumo system}, or the FHNcyl system for short.
Taking formally $\cS=\R$ in Eq.~\eqref{eq:FHN} gives Eq.~\eqref{FHN}.

A solution to Eq.~\eqref{FHN} which is a function of a 
single variable, $z=x-ct$, $c>0$, and vanishes at infinity is 
called a \textbf{pulse}. One of the first results
on the existence of pulses 
is due to Hastings~\cite{Hastings_76},
who showed that when $\cS$ is the real line,
$0<\alpha<\frac12$ and $\eps$, $\gamma$ are
positive and sufficiently small, 
Eq.~\eqref{FHN} has a pulse solution,
whose speed depends on $\alpha$, $\gamma$, and $\eps$.
The pulse is obtained as a homoclinic orbit in a related 
system of ordinary differential equations.

It turns out that when $\eps>0$ is sufficiently small, 
Eq.~\eqref{FHN} 
 has at least two 
different pulse 
solutions, 
the \textbf{fast} pulse studied by Hastings,
which travels with speed 
$c_f(\eps)= \frac{\sqrt{2}}{2}(1-2\alpha)+o(\eps)$, and a \textbf{slow} pulse that travels with speed $c_s(\eps) = O(\sqrt{\eps})$ (Carpenter~\cite{Carpenter_77}, Hastings~\cite{Hastings_76, Hastings_82}, Langer~\cite{Langer_80}, Krupa, Sandstede and Szmolyan~\cite{KSS97}, Jones, Kopell and Langer~\cite{JKL91}, Arioli and Koch~\cite{AK15}). Langer~\cite{Langer_80} also proved uniqueness of the fast pulse.  Jones \cite{Jones_84}, and independently, Yanagida~\cite{Yanagida_85}, proved that the fast pulse is stable. 
 In addition, fast pulses with oscillatory tails exist and are stable (Carter and Sandstede~\cite{CS15}, Carter, de Rijk and Sandstede~\cite{CdRS16}). On the other hand, the slow pulse is always unstable (Flores~\cite{Flores}, Evans~\cite{EvansIV}, Ikeda, Mimura and Tsujikawa~\cite{IMT87}).

Existence and stability for fast pulses have been studied for 
variants of Eq.~\eqref{FHN}, where the second equation also has
a diffusion term (Cornwell and Jones~\cite{Cornwell-Jones18}, 
Chen and Choi~\cite{Chen-Choi15}, Chen and Hu~\cite{Chen-Hu14}). 
Another system that admits stable fast pulses is the 
discrete analogue of Eq.~\eqref{FHN} (Hupkes and Sandstede~\cite{HS13}, 
Schouten-Straatman and Hupkes~\cite{SSH19}, Hupkes, 
Morelli, Schouten-Straatman and Van Vleck~\cite{HMSSVV}).

There are a few results in higher dimensions. In $\R^2$, 
Mikhailov and Krinskii~\cite{MK83} and Keener~\cite{Keener86} 
studied spiral solutions of 
Eq.~\eqref{FHN}. In $N$-dimensions, Tsujikawa, Nagai, 
Mimura, Kobayashi and Ikeda~\cite{Tsujikawa_89} proved that there 
exist fast pulse solutions propagating in a one-dimensional direction. 
Such solutions are stable. 

In this paper we study solutions of the FHNcyl system, Eq.~\eqref{eq:FHN}, on 
infinitely long, thin cylindrical surfaces.
For $\cS$ a standard cylinder, the pulse solutions to Eq.~\eqref{FHN}  
are also (angle-independent) solutions to Eq.~\eqref{eq:FHN} 
and we continue to call them the pulses.
We show that
\begin{enumerate}[{\em (i)}]
	\item on a cylinder of small constant radius,
the (fast) pulses 
are stable under general perturbations of the initial condition
that depend on both spatial variables;
	\item on a warped cylinder whose radius 
is small  and varies slowly along its length, 
solutions that are initially close to
a pulse stay near the family of pulses for all time.
\end{enumerate}
Our extension, Eq.~\eqref{eq:FHN}, of Eq.~\eqref{FHN}  is geometrical rather 
than biophysical. However, 
the techniques we use are fairly robust and could be easily modified for 
more realistic second order elliptic operators describing the 
surface evolution instead of the Laplacian $\Del_{\cS}$.

\subsection{Main results}\label{sec:results}

Consider Eq.~\eqref{eq:FHN} on the standard 
cylinder of constant radius $R$ centered about the $x$-axis in $\R^3$,
$$
\cS_R := \left \{ (x,R\cos\theta, R\sin\theta) 
\in\R^3\ \big\vert\ 
x\in \R,\, \theta\in S^1
\right\}\,,
$$
where $S^1=\R/(2\pi)$ is the unit circle. 
The Laplacian on this surface is defined by
$$
\Delta_{\cS_R}= \partial_x^2 + R^{-2}\partial_\theta^2\,,
$$
and the Riemannian area element is $R\,d\theta dx$.
Clearly, the cylindrical FitzHugh-Nagumo system on
$\cS = \cS_{R}$ is invariant under translations. If
$u(x, \theta, t)$ is a solution, then so are its translates
$$
u_h(x, \theta, t) := u(x - h, \theta, t)\,,\quad h\in \R.
$$

Each pulse $\Phi$ on $\cS=\R$ defines
a smooth axisymmetric traveling 
wave solution $u(x,\theta,t)=\Phi(x-ct)$ of Eq.~\eqref{eq:FHN}
on $\cS_R$. Its speed $c$ is determined by
the parameters $\alpha$, $\gamma$, and~$\eps$.
It is a consequence of translation invariance
that all translates $\Phi_h$ of $\Phi$ are pulses
of the same speed $c$.

\bigskip Our first result concerns the stability of a 
particular fast pulse, $\Phi$, under perturbations of the
initial values that need not be axisymmetric. 
We consider mild solutions, defined by an integral equation
derived from Duhamel's formula, in the mixed Sobolev space
$$
H^{2,1}
:=\left\{u=(u_1, u_2)\in L^2(\cS_R)\!\times\!L^2(\cS_R)\ \big\vert\ 
\Delta_{\cS} u_1\in L^2(\cS_R), \p_xu_2\in L^2(\cS_R)\right\}\,.
$$
The norm on this space is denoted by $\|\cdot\|_{2,1}$.
With this notion of solution,
the initial-value problem is locally well-posed, i.e., 
for each initial value $u_0\in H^{2,1}$,
a unique mild solution exists on some positive time interval,
and this solution depends continuously on the 
initial data, see Proposition~\ref{prop:LWP}.   The following
theorem says that mild solutions which are initially close to $\Phi$ 
approach nearby translates of~$\Phi(x-ct)$ as $t\to \infty$.
In technical terms, the traveling pulse
$\Phi(x-ct)$ is orbitally asymptotically stable. 

\begin{theorem}[Stability of pulses, standard cylinder]
\label{thm:nonli-stab}
Consider Eq.~\eqref{eq:FHN} 
on the cylinder $\cS_R$ of constant radius $R\le 1$.
Fix $\alpha\in (0,\frac12)$, $\eps>0$, and $\gamma>0$ such that
the equation has a fast pulse solution $\Phi(x-ct)$.
If $\eps$ is sufficiently small, then there is
a neighborhood $\cU$ of $\Phi$ in $H^{2,1}$ such that
for every $u_0\in \cU$,
the mild solution $u(t)$ with initial value $u\big\vert_{t=0}=u_0$
exists globally in time and satisfies
\begin{equation}\label{eq:nonli-stab}
\|u(t)- \Phi_{ct+h_*} \|_{2,1} 
\le C_1 e^{-\xi t} \|u_0-\Phi\|_{2,1} \qquad (t\ge 0)
\end{equation}
for some $\xi>0$ and $h_*\in\R$ (determined by $u_0$)
with $$|h_*|\le C_2\|u_0-\Phi\|_{2,1}\,.
$$
Here, $C_1$ and $C_2$ are positive constants.
\end{theorem}
\noindent 

Theorem~\ref{thm:nonli-stab} is proved in  Section~\ref{sec:nonli-stab}.  Though, our proof of Theorem~\ref{thm:nonli-stab} could be slightly shortened by appealing to the general Theorem 4.3.5 of \cite{KP13}, for the reader's convenience, we provide a self-contained 
proof which uses only a spectral result of~\cite{Jones_84} and~\cite{Yanagida_85}, in addition to well-known results about semigroups.
Under the assumptions of the theorem, Eq.~\eqref{eq:nonli-stab} holds
for any decay rate  $\xi$ with $\xi< \min\{\alpha,\beta,\eps\gamma\}$, 
where $\beta$ is the exponent from Lemma~\ref{lem:L0}.
The neighborhood $\cU$, as well as
the constants $C_1$, $C_2$ and $\xi$ depend only on the parameters
$\alpha$, $\gamma$, and $\eps$.

The translates of $\Phi$ form a one-dimensional manifold of pulses
\begin{equation}\label{eq:manifold}
\cM := \{ \Phi_h \;\vert\; h\in \R \}\,.
\end{equation}
Denote by $\dist(v, \cM):=\inf_h \| v \!-\! \Phi_h \|_{2,1}$
the distance of $v\in H^{2,1}$ from the manifold.
By translation invariance, the conclusion of
Theorem~\ref{thm:nonli-stab}
yields a tubular neighborhood
$\cW= \{w\in H^{2,1} \mid \dist(w,\cM)<\eta\}$ such that 
$$\dist(u(t),\cM)\le C_1e^{-\xi t} \dist(u_0,\cM)$$
for all mild solutions with initial values in $\cM$.
As $t\to\infty$, each solution converges to a 
particular traveling pulse $\Phi(x-ct-h_*)$.
\bigskip

For our second result, we consider
{\bf warped} cylindrical surfaces, 
defined as graphs over the standard one, with a variable radius
$\rho(x)$,
\begin{equation}
\label{eq:def-Srho}
\cS_\rho := \left \{(x,\rho(x)\cos\theta, \rho(x)\sin\theta)
\in\R^3 \ 
\big\vert\ 
x\in \R,\theta\in S^1 \right\}\,.
\end{equation}
On $\cS_\rho$, the Laplace-Beltrami operator 
is given by
\begin{equation}\label{eq:Delta-rho}
\Del_{\cS_{\rho}} := 
\frac{1}{\sqrt{g(x)}} \p_x
\bigg( \frac{\rho(x)}{\sqrt{1 + \rho'(x)^2}} \p_x \bigg)
+ \rho^{-2}(x) \p^2_{\theta}\,,
\end{equation}
where $g(x):=\rho(x)^2(1 + \rho'(x)^2)$ is the squared 
Riemannian density.

We identify functions $u$ on $\cS_\rho$
with functions on $\cS_R$ via the coordinate
map \begin{equation}
\label{eq:diffeo}
\psi_\rho(x,R\cos\theta, R\sin\theta)
= \bigl(x, \rho(x)\cos\theta,\rho(x)\sin\theta\bigr)
\end{equation}
from $\cS_R$ to $\cS_\rho$.
Under the assumption that $\rho$ is twice continuously differentiable,
positive, bounded, and bounded away from zero, 
$\psi_\rho$  is a diffeomorphism of class $C^2$.
Via this identification, the norms $\|\cdot\|$ and
$\|\cdot\|_{2,1}$ on the standard cylinder are pushed 
forward to apply to functions on $\cS_\rho$.

When $\rho$ is non-constant, Eq.~\eqref{eq:FHN} cannot be expected
to have pulse solutions.  However, if
$\rho$ is almost constant, then there are
near-pulse solutions that remain in a neighborhood of $\cM$
for all time. Generally,
these solutions do not stay close to any particular pulse, but
move slowly along the manifold:

\begin{theorem}[Near-pulse solutions, warped cylinder]
	\label{thm:persist}
Consider the FHN system
on a cylinder $\cS_\rho$ of variable radius, as in 
Eq.~\eqref{eq:def-Srho}, and let $\alpha$, $\eps$ and  
$\gamma$ be as in Theorem~\ref{thm:nonli-stab}. 
There are a constant $\delta_*>0$ and a tubular
neighborhood $\cW$ of $\cM$ in $H^{2,1}$, with the following properties:
If $R\le 1$ and $\delta:=R^{-1}\|\rho-R\|_{C^2}\le \delta_*$,
then for every $u_0\in\cW$, the unique mild
solution $u(t)$
with initial value $u\big\vert_{t=0}=u_0$
exists globally in time, and satisfies
\begin{equation}\label{eq:persist}
	\dist(u(t), \cM) \leq C_1 e^{-\xi t}\dist(u_0,\cM) + C_2\delta\,,
\qquad (t\ge 0)
	\end{equation}
for some positive constants $C_1$, $C_2$, and $\xi$.
\end{theorem}
Under the assumptions of the theorem,
initial values $u_0\in\cM$ give rise
to near-pulse solutions that satisfy 
$\sup_{t>0}\dist(u(t),\cM)\le C_2\delta$.
The Theorem~\ref{thm:persist} will be proved in Section~\ref{sec:warped}.
The tubular neighborhood $\cW$, as well
as the constants $\delta_*$, $C_1$, $C_2$, and $\xi$ depend only on 
the parameters $\alpha$, $\gamma$, and $\eps$.
In general the value of the exponential decay rate
$\xi$ in Eq.~\eqref{eq:persist} lies below 
that of Eq.~\eqref{eq:nonli-stab}.

\subsection{Outline of the arguments} The key
idea is to write a solution near a pulse $\Phi$ as
the superposition of a modulated pulse with a transversal 
fluctuation.  On the standard cylinder $\cS_R$,
Theorem~\ref{thm:nonli-stab} concludes that
(for suitable values of the parameters) the fluctuation decreases 
exponentially over time, as the solution 
settles on a nearby translate of~$\Phi$.
In the proof, we first establish
linearized stability (Section~\ref{sec:lin-est}),
and then apply a fixed point argument for
the nonlinear evolution (Section~\ref{sec:nonli-stab}). 

On the warped cylinder $\cS_\rho$, 
Theorem~\ref{thm:persist} provides bounds
on the fluctuation of near-pulse solutions
in terms of the distance of the variable radius $\rho$ 
from the constant $R$. In the proof,
we use the standard cylinder for reference,
and combine a perturbation result for the linearization
with the exponential orbital stability proved in Theorem~\ref{thm:nonli-stab}.

We briefly describe some key technical aspects of the proofs.
Given a fast pulse~$\Phi$, let $c$ be its speed,
and consider the FHNcyl system on the standard 
cylinder of radius $R$. We employ a moving frame
($z=x-ct$),
where the pulse is stationary.  In this frame 
the FHNcyl system becomes
\begin{equation}\label{eq:FHN_moving}
\begin{split}
\p_t u_1 &= \Delta_{\cS_R} u_1 + c \p_z u_1 + f(u_1) - u_2\\  
\p_t u_2 &= c \p_z u_2 + \eps (u_1 - \gamma u_2)\,.
\end{split}
\end{equation}

Denote by $L$ the linearization of Eq.~\eqref{eq:FHN_moving}
about the stationary solution $\Phi$. 
In Section~\ref{sec:lin-est}, we prove that 
the semigroup generated by $L$ decays exponentially 
in directions transversal to 
the tangent space of $\cM$ at $\Phi$, that is,
\begin{equation*}
\| e^{tL}(1-P) \|_{2,1} \lesssim e^{-\sigma t}
\qquad (t\ge 0)\,,
\end{equation*}
for some positive constant $\sigma$.
Here, $P$ is a projection onto the tangent space
that commutes with $L$, and $\|\cdot\|_{2,1}$ denotes
the operator norm on $H^{2,1}$.
The proof of this estimate is the most challenging part of the paper.  
Since $L$
is not self-adjoint, $P$ is not an orthogonal projection.
Noting that the tangent space of $\cM$ at $\Phi$ 
is spanned by the derivative $\tau:=-\p_z \Phi$, 
and that $L\tau=0$, we construct $P$ as the spectral 
projection associated with the zero eigenvalue of $L$.
The decay estimate follows from the fact
that the remainder of the spectrum lies in the left 
half-plane. 

In Section~\ref{sec:nonli-stab}, we establish the nonlinear stability
of the pulse (cf.  \cite{KP13}) and prove Theorem~\ref{thm:nonli-stab}.
We work again in the moving frame, and
show that every mild solution of Eq.~\eqref{eq:FHN_moving}
that starts out sufficiently close to $\Phi$ converges
exponentially to a translated pulse~$\Phi_{h_*}$. 
As explained above, we decompose such a solution into a modulated pulse 
(moving on $\cM$) and a fluctuation (transversal to $\cM$). 
Concretely, we write $u= \Phi_{h}+v$, choosing $h$ in such a way that
$$
\Phi_{h}-\Phi \approx P (u-\Phi)\,,\qquad v \approx
(1-P)(u-\Phi)
$$
with errors of order $\|u-\Phi\|_{2,1}^2$
in a neighborhood of $\Phi$.  This transforms 
Eq.~\eqref{eq:FHN_moving} into an equation for
$v(t)$ in the canonical form
$$
\p_t v = Lv + N(v,h) \,,
$$
where $N(v,h)$ is of order 
$(|h|+\|v\|_{2,1})\,\|v\|_{2,1}$, coupled to an ordinary differential
equation for the evolution of $h$.
With the help of the linearized decay estimate 
from Section~\ref{sec:lin-est}, we prove the bounds
$$
\|v(t)\|_{2,1}\lesssim e^{-\xi t}\|v_0\|\,,\qquad
|h(t)-h(0)|\lesssim e^{-\xi t}\|v_0\|^2\qquad
(t\ge 0)
$$
for any $\xi$ with $\xi<\sigma$, i.e., the fluctuation
decays exponentially, while the modulation converges.
Since $\|v_0\|_{2,1}\lesssim \|u_0-\Phi\|_{2,1}$
and $h_0\lesssim \|u_0-\Phi\|_{2,1}$,
this yields the conclusion of 
Theorem~\ref{thm:nonli-stab}.

Section~\ref{sec:warped} contains the
proof of Theorem~\ref{thm:persist}. We consider the
variable radius $\rho(x)$ of a warped cylinder 
as a perturbation of $R$, and then appeal to 
Theorem~\ref{thm:nonli-stab}.  The basis for the 
argument is an estimate for the linearized evolution
on $\cS_\rho$.
Note that in the moving frame which we used on the standard cylinder,
the variation of the radius amounts to a {\em time-dependent}
perturbation of the principal part,
as $\Delta_{\cS_R}$ becomes $\Delta_{\cS_{\rho(z+ct)}}$. 
To avoid this issue, we study the perturbation
in the static frame, and linearize the FHNcyl system
about the zero solution instead of the traveling
pulse.  It turns out that the
linearized operator, $A_\rho$, is sectorial.
Therefore, we can represent the semigroup $e^{t A_\rho}$ by
an absolutely convergent contour integral, and control the
perturbation via resolvent estimates.
We use Gr\"onwall's inequality to extend these 
perturbation estimates to the nonlinear evolution 
generated by the FHNcyl system on $\cS_\rho$. In combination with
the exponential decay of fluctuations
for near-pulse solutions on $\cS_R$
that was proved in Theorem~\ref{thm:nonli-stab}, this yields
Theorem~\ref{thm:persist}.

\subsection{Preliminaries and notation}

We make the standing assumption that $\rho$ is of class $C^2$,
positive, bounded, and bounded away from zero.
Also, we assume that $\alpha\in (0,\frac12)$, and
that $\gamma>0$, $\eps>0$ are small enough that
the FHNcyl system~\eqref{eq:FHN} admits a 
fast pulse solution, $\Phi$.
Furthermore, we assume that $\eps$ is small enough so that
the spectral results of~\cite{Jones_84} and~\cite{Yanagida_85} apply.

On a cylindrical surface $\cS_\rho$, the FHNcyl system takes the form
\begin{equation}\label{eq:FHN_static}
\begin{split}
\p_t u_1 &= \Delta_{\cS_{\rho}} u_1 + f(u_1) - u_2 \\
\p_t u_2 &= \eps (u_1 + \gamma u_2)\,,
\end{split}
\end{equation}
where the Laplace-Beltrami operator $\Delta_{\cS_{\rho}}$
is given by Eq.~\eqref{eq:Delta-rho}. 
Denote the right hand side of Eq.~\eqref{eq:FHN_static}
by $F_\rho(u)$. 

We consider the 
initial-value problem
\begin{equation}\label{eq:ivp-stat}
\p_t u = F_\rho(u), \qquad u\vert_{t=0} = u_0
\end{equation}
on the space $H^{2,1}$.
The principal part of $F_{\rho}$ is given by its
G\^ateaux derivative
\begin{equation} \label{eq:A} 
A_\rho:= dF_\rho(0)
= \begin{pmatrix} \Delta_{\cS_{\rho}} -\alpha & -1 \\
\eps & - \eps\gamma 
\end{pmatrix}\,.
\end{equation}
Since $A_\rho$ is a bounded perturbation
of the diagonal operator that acts as $\Delta_{\cS_\rho}$
on the first component, and vanishes on the second,
it generates an analytic semigroup
$e^{tA_\rho}$ in $L^2$. We will show in Lemma~\ref{lem:domain-A} 
that the semigroup restricts to a uniformly bounded
analytic semigroup on the dense invariant subspace $H^{2,1}$. 

We start out by verifying that the initial-value problem
for Eq.~\eqref{eq:FHN_static} is well-posed 
locally in time.  Since $F(0)=0$, we
can expand $F(u)= \p_t u = A_{\rho} u + N(u)$,
where $A_{\rho} = dF_\rho(0)$ was defined in 
Eq.~\eqref{eq:A}, and
\begin{equation}
\label{eq:N}
N(u) := F_\rho (u) - A_{\rho} u = 
\begin{pmatrix}
-u_1^3 +(\alpha+1)u_1^2 \\
0
\end{pmatrix}.
\end{equation}
If $u(t)$ is a classical solution of
Eq.~\eqref{eq:FHN}, then by Duhamel's
formula it also solves the integral equation
\begin{equation}\label{eq:Duhamel-F}
u(t)= e^{tA_\rho}u_0 + \int_0^te^{(t - s)A_\rho}N(u(s))\, ds
=: \cF_\rho(u)(t)\,. 
\end{equation}

By definition, a {\bf mild solution} of \eqref{eq:ivp-stat}
is a strongly continuous function $u(t)$ taking
values in $H^{2,1}$ that solves the fixed point problem
$u=\cF_\rho(u)$ in the space $C([0, T]; H^{2,1})$ 
for some $T>0$. Note that since $H^{2,1}$
is contained in the domain of $L$, the time derivative of a 
mild solution lies in $L^2$ and satisfies
Eq.~\eqref{eq:FHN} in $L^2$. Moreover, smooth initial
values rise to classical solutions.

Since the cylindrical surface has dimension 2, 
the Sobolev space $H^2(\cS_R)$ is a Banach algebra.
It follows directly from the continuity of the multiplication
that $N$, as a polynomial in the first component, 
is locally Lipschitz in $H^{2,1}$. Explicitly, for
every $\eta>0$ there exists a constant $C_\eta>0$ 
(which depends on $\alpha$, $\gamma$, and $\eps$, as well as $\eta$)
such that 
\begin{equation}
\label{eq:Lip-N}
\| N(u) - N(w) \|_{2,1}  \leq C_\eta
\| u_1 - w_1 \|_{H^2}
\end{equation}
for all $u,w$ with $\|u_1\|_{H^2},\|w_1\|_{H^2}\le \eta$.
From here, local well-posedness of the
initial-value problem follows by standard
methods.  For the sake of completeness, we construct
the mild solution of Eq.~\eqref{eq:FHN_static}, as follows.

\begin{proposition}[Local well-posedness]
	\label{prop:LWP}
Assume that $\rho$ is of class $C^2$, bounded, and bounded away
from zero. Then for each $u_0\in H^{2,1}$, 
there exists $T>0$ (depending on $\|u_0\|_{2,1}$)
such that Eq.~\eqref{eq:FHN_static}  on $\cS_\rho$
has a unique mild solution $u$ in $C([0,T],H^{2,1})$
with initial condition $u\vert_{t=0}=u_0$.
The solution depends continuously on $u_0$.
\end{proposition}

\begin{proof}
We proceed by Picard iteration.
Given $u_0\in H^{2,1}$, fix $\eta>0$ and $T>0$ (to be
specified below)  and consider
$$
\cB:= \left\{ u\in C([0,T],H^{2,1})\ \big\vert\ 
\|u(t)\|_{2,1}\le \eta\ \text{for all}\ 0\le t\le T\right\}\,,
$$
equipped with the norm $\|u\|_T:=\sup_{0\le t\le T} \|u(t)\|_{2,1}$.

The map $\cF_\rho$ defined by Eq.~\eqref{eq:Duhamel-F}
is Lipschitz continuous on $\cB$,
\begin{align*}
\| \cF_\rho(u)- \cF_\rho(w)\|_T
& \le \sup_{0\le t\le T} \int_0^t \|e^{(t-s)A_\rho}\|_{2,1}
\|N(u(s))-N(w(s))\|_{2,1}\, ds\\
&\le C_0C_\eta T\,\|v-w\|_T\,,
\end{align*}
where $C_0:=\sup_{t\ge 0} \|e^{tA_\rho}\|_{2,1}$,
and $C_\eta$ is as in Eq.~\eqref{eq:Lip-N}.
Moreover,
$$ \|\cF_\rho(0)\|_T =\sup_{0\le t\le T} \|e^{tA_\rho}u_0\|_{2,1}
\le C_0\, \|u_0\|_{2,1}\,. 
$$

Choose $\eta=2C_0\|u_0\|_{2,1}$, and $T=(2C_0C_\eta)^{-1}$.
Then $\cF_\rho$ has Lipschitz constant $\frac12$
and maps $\cB$ into itself. By Banach's contraction
mapping theorem, $\cF_\rho$ has a unique fixed point in $\cB$,
which provides the desired mild solution of 
Eq.~\eqref{eq:FHN_static}.

Let $w(t)$ be another mild solution, whose
initial value $w_0:=w\big\vert_{t=0}$ satisfies 
$\|w_0\|_{2,1}<\eta$. The difference between the 
solutions is bounded by
\begin{align*}
\|u(t)-w(t)\|_{2,1}
&\le \|e^{tA_\rho}(u_0-w_0)\|_{2,1}
+\int_0^t \|e^{(t-s)A_\rho}\bigl(N(u(s))-N(w(s))\bigr)\|_{2,1}\, ds\\
&\le C_0\|u_0-w_0\|_{2,1}
+ C_0C_\eta\int_0^t\|u(s)-w(s)\|_{2,1}\, ds\,,
\end{align*}
so long as $\max\{\|u(s)\|_{2,1}, \|w(s)\|_{2,1}\}\le \eta$
for all $0\le s\le t$. By Gr\"onwall's inequality,
$$
\|u(t)-w(t)\|_{2,1}\le C_0 e^{C_0C_\eta t}\|u_0-w_0\|_{2,1}\,.
$$
This proves continuous dependence on initial data.
\end{proof}

Consider now the FHNcyl system on the standard cylinder 
$\cS_R$. By translation invariance, the linearization
$A_R$ commutes with translations 
in space and time. In the moving frame,
as described by Eq.~\eqref{eq:FHN_moving},
the principal part of the system is given by $\bar L:=A_R + 
c\partial_z$.
As a sum of commuting operators, $\bar L$ 
generates a semigroup $e^{t\bar L}$ on $L^2$ that acts as
$$
\bigl(e^{t\bar L}u\bigr)(z) = \bigl(e^{tA_R}u\bigr)(z+ct)\,.
$$
Like the group of translations, the semigroup $e^{t\bar L}$
is strongly continuous but not analytic.
Since translations are isometries of $H^{2,1}$,
$ \|e^{t\bar L}\|_{2,1}=\|e^{tA}\|_{2,1}$ for all
$t>0$. In particular, mild solutions
of Eq.~\eqref{eq:FHN_moving} are equivalent
to mild solutions of Eq.~\eqref{eq:FHN} on $\cS_R$ 
via the transformation $z=x-ct$.

\bigskip A general remark on the use of constants:
In our estimates, it is understood that
constants may vary from equation to equation, 
and depend on the fixed parameters $\alpha$, $\gamma$, $\eps$
of Eq.~\eqref{eq:FHN}. 
We frequently use the notation $\lesssim$ and $\gtrsim$
for the respective inequalities up to  
such constants. Dependence on other parameters,
including $R$, $\rho$, and $u_0$ will be made explicit.


\subsection*{Acknowledgements}
It is a pleasure to thank Mary Pugh and Adam Stinchcombe for many stimulating discussions and the third author is grateful to Daniel Sigal for very helpful discussions of the mechanism of propagation of pulses in axons. 
 The research for this paper
 was supported in part by NSERC through Discovery Grant No. 311685 (A.B.) and Grant No. NA7901
 (I.M.S.).

\section{Linearized stability on the standard cylinder}
\label{sec:lin-est}

Fix a fast pulse $\Phi$ on $\cS_R$,
and let $\cM$ be the manifold of its translates 
defined in Eq.~\eqref{eq:manifold}. 
By definition, $\Phi$ is an axisymmetric traveling
wave solution of Eq.~\eqref{eq:FHN}.
In this section we study the linear stability of $\Phi$.

In the moving frame where $\Phi$ is stationary, the system
is given by Eq.~\eqref{eq:FHN_moving}. Let
$G(u)=F_R(u)+\partial_z u$ be the right hand side of this equation.
Since the pulse is a stationary solution, $G(\Phi) = 0$.
The linearization about $\Phi$
is given by the G\^ateaux derivative
\begin{equation} \label{eq:L} 
L:= dG(\Phi)
= \begin{pmatrix} \Delta_{\cS_R} + c\p_z + f'(\phi_1) & -1 \\
\eps & c\p_z - \eps\gamma 
\end{pmatrix}\,.
\end{equation}

The linearization defines a closed linear operator
on the Hilbert space
$L^2 :=L^2(\cS_R; \bC^2)$ of two-component
square integrable functions, with the 
inner product 
\begin{equation}\label{eq:eps-inner}
\langle u, w\rangle :=
\int_\R \int_{S^1} \left( u_1\bar w_1 +\eps^{-1} u_2\bar w_2\right) 
R \,d\theta  dz\,,
\end{equation}
and the corresponding norm $\|\cdot\| $.
The domain of $L$ is the dense subspace 
\begin{align}\label{eq:H21}
H^{2,1}
:=\left\{u=(u_1, u_2)\in L^2
\ \big\vert\  
\|u\|_{2,1} := \| (\Delta_{\cS_R} u_1,  \p_z u_2)
\| + \|u\|<\infty \right\}\,,
\end{align}
whose norm $\|u\|_{2,1}$ is equivalent to the graph norm of $L$,
see Lemma~\ref{lem:L-domain}. Note that $H^{2,1}$ properly contains
$H^2\times H^1$, because Eq.~\eqref{eq:H21}
does not require $\partial_\theta u_2$ to lie in~$L^2$. 
We will show that $L$ generates a strongly continuous 
semigroup $e^{tL}$.

Since $G(\Phi)=0$, by translation invariance 
we have that $G(\Phi_h)=0$ for any $h\in \R$. 
Differentiating this equation with respect to $h$,
we find that $L\p_z\Phi =0$.
This means that 0 is an eigenvalue of $L$, and
the tangent vector $\tau=-\partial_z\Phi$ 
is an eigenfunction.  It turns out that the spectral 
projection $P$ associated with
the zero eigenvalue of $L$ has rank one.
The complementary projection $Q=1-P$
is the projection onto the range of $L$.
Both $P$ and $Q$ commute with $L$ and with the semigroup $e^{tL}$.


The main result of this section is the following:
\begin{proposition}[Linearized decay]
	\label{prop:LQ-decay} 
Let $L$ be the operator defined by Eq.~\eqref{eq:L}.
If $\eps>0$ is sufficiently small, and $0<R\le 1$,
then there exists $\sigma>0$ such that the semigroup 
$e^{tL}$ satisfies 
	\begin{equation}\label{eq:LQ-decay}
	\big\| e^{tL} Q\big\|_{2,1} \leq Ce^{-\sigma t},
	\qquad (t\ge 0) 
	\end{equation} for some constant $C$.
\end{proposition}

\noindent As an immediate consequence,
the semigroup $e^{tL}$ is uniformly bounded.
Indeed, by the triangle inequality,
$$
\|e^{tL}\|_{2,1}\le \|e^{tL}P\|_{2,1} +  \|e^{tL}Q\|_{2,1}
\le \|P\|_{2,1}+C\|Q\|_{2,1} \qquad (t\ge 0)\,,
$$
since $e^{tL}$ is constant on the range of $P$ and decays exponentialy
on the range of $Q$.

The proof of Proposition~\ref{prop:LQ-decay}
will be given 
in Subsection~\ref{subsec:LQ-decay}, and
will specify the conditions on $\sigma$.  An important
tool is the following general result of 
Pr\"uss~\cite[Corollary 4]{Pruss} that we state next:

\begin{theorem}[Pr\"uss]\label{thm:Pruss}
	Suppose that $B$ 
generates a strongly continuous semigroup on a Hilbert
space. If the resolvent $(\lambda-B)^{-1}$ is
uniformly bounded on the half-plane
$\{\lambda\in\bC\mid Re\, \lambda\ge -\beta \}$, for some $\beta > 0$, then 
there exists a constant $C>0$ such that
$\| e^{tB} \| \leq Ce^{-\beta t}$ for all $t\ge 0$.
\end{theorem}

\subsection{The linear semigroup}
\label{subsec:semigroup}

We first justify the choice of the function space $H^{2,1}$. 

\begin{lemma} [Domain of $L$]
	\label{lem:L-domain}
	The operator $L$ defined in Eq.~\eqref{eq:L}
	has domain $D(L)=H^{2,1}$, and its graph norm satisfies
$$
	\|u\|_{2,1}\lesssim \|Lu\|+ \|u\|\lesssim \|u\|_{2,1}\,,\quad
	u\in H^{2,1}\,.
$$
\end{lemma}
\begin{proof} 
	We use that 
$ \|\p_z u_1\|\le \tfrac{1}{2c} \|\Delta_{\cS_R}u_1\| +\tfrac{c}{2} \|u_1\|$,
and set 
	\begin{equation}
	\label{eq:b}
	b:=\sup_{z\in\R} |f'(\phi_1(z))-f'(0)|<\infty\,.
	\end{equation}
	By the reverse triangle inequality, this yields for the lower
	bound
	$$ \|Lu\| \ge 
	\min\left\{\tfrac12, c\right\}\|(\Delta_{\cS_R}u_1,\partial_z u_2)\|
	-(\tfrac{c^2}{2}+b)\|u_1\|\,,
	$$
	which implies that
	\begin{align*}
	\left(1+ \tfrac{c^2}{2}+b\right)
	( \|Lu\|+\|u\|) 
	&\ge \|Lu\| + \left(1+ \tfrac{c^2}{2}+b\right)\|u\|\\
	&\ge \min\left\{\tfrac12, c\right\}\|u\|_{2,1}\,.
	\end{align*}
	For the upper bound, the triangle inequality
	yields 
	$$
	\|Lu\|+\|u\|\le \max\left\{\tfrac32 , c, 1+\tfrac{c^2}{2}+b\right\} 
	\|u\|_{2,1}\,.
	\\[-0.5cm]
	$$
\end{proof}

An operator $B$ on a Hilbert space is called 
{\bf dissipative}, if $\Re\,\langle Bx,x\rangle \le 0$
on its domain.  We will frequently 
use the following convenient corollary of the 
Lumer-Phillips theorem (see \cite[Theorem 1.4.3]{Pazy}).

\begin{lemma}[Dissipative operators]
	\label{lem:dissipative}
	Let $B$ be a closed, densely defined,
	dissipative operator on a Hilbert space $H$.
	Then $B$ generates a strongly continuous
	semigroup of contractions, $e^{tB}$.
	Its spectrum lies in the left half-plane
	$\{\lambda\in\bC\mid\Re\,\lambda\le 0\}$, and
	its resolvent is bounded by
	$$
	\|(\lambda-B)^{-1}\|\le\frac{1}{\Re\,  \lambda}\,,
	\qquad (\Re\,\lambda>0)\,.
	$$
\end{lemma}

\begin{proof} Since $B$ is dissipative, the operator $1-B$ is injective,
	$$ \|(1-B)v\|\, \ge \|v\|^{-1}\, |\Re\, \langle (1-B)v,v\rangle|
	\ge \|v\|\,.
	$$ 
	Since its adjoint is injective by the same argument,
	the range of $1-B$ is dense.
	
	Let $w_0\in H$ be arbitrary. 
Choose a sequence $(w_n)$ in the range of $1-B$
with $\lim w_n=w_0$. 
For each $n$, let $v_n$ be in the domain of $B$ such 
that $(1-B)v_n=w_n$.  Since
	$$
	\|v_n-v_m\|\le \|(1-B)(v_n-v_m)\|
	= \|w_n-w_m\|\,,
	$$ 
by the Cauchy criterion the sequence $(v_n)$ converges 
to some limit, $v_0$. Since $B$ is a closed operator, $w_0=(1-B)v_0$.
	We conclude that $1-B$ is surjective.
	
	By the Lumer-Phillips theorem, 
	$B$ generates a
	strongly continuous semigroup of contractions on $H$.
	The resolvent bound follows from the 
	Hille-Yosida theorem (see \cite[Theorem 1.3.1]{Pazy}).
\end{proof}

The next lemma will be used to construct the
semigroup generated by $L$. It also plays a role
in the spectral estimates.

\begin{lemma}[Compact perturbation]
\label{lem:compact}
	The operator $L$ is a bounded, relatively compact perturbation of 
	\begin{equation}\label{eq:Lbar} \bar L:= 
	\begin{pmatrix}
	\Delta_{\cS_R} + c\p_z  -\alpha & -1 \\
	\eps & c\p_z - \eps\gamma 
	\end{pmatrix}\,.
	\end{equation}
\end{lemma}
\begin{proof} We decompose  $L=\bar L+V$, where
	$V$ is the matrix multiplication operator
$$
V:= \begin{pmatrix}
	f'(\phi_1)-f'(0) & 0 \\
	0 & 0 \end{pmatrix}\,.
$$
	Since $f$ is a polynomial and $\phi_1$ is a smooth, bounded function,
	$V$ is a bounded operator on $L^2$ and also on $H^{2,1}$.
	The term $f'(\phi_1(z))-f'(0)$ is continuous and decays 
	at infinity. By a standard result 
    (see \cite[Theorem, 3.1.11]{KP13}), $V\bar L^{-1}$ is compact.
\end{proof}

The operator $\bar L$ captures the behavior of 
$L$ as $z\to\pm \infty$. The next lemma implies that
$\bar L$ generates a strongly continuous semigroup of contractions.

\begin{lemma}[Principal part of $\bar{L}$] \label{lem:Lbar}
	Let $\bar L$ be the operator defined by Eq.~\eqref{eq:Lbar},
	and $\sigma=\min\{\alpha,\eps\gamma\}$. 
Then $\Re\,\langle \bar L v, v\rangle \le -\sigma\|v\|^2$.
\end{lemma}

\begin{proof} By Lemmas~\ref{lem:L-domain} and~\ref{lem:compact},
	the domain of $\bar L$ is $H^{2,1}$.
	Since $\partial_z$ is skew-adjoint in $L^2(\cS_R)$,
	and the off-diagonal terms in $\bar L$ are
	skew-adjoint with respect to the inner
	product from Eq.~\eqref{eq:eps-inner}, we have
\begin{align*} 
\Re\, \langle \bar{L} v, v \rangle
&= \int_{\cS_R}
	((\Delta_{\cS_R}v_1)\bar v_1 -\alpha |v_1|^2 - 
	\gamma |v_2|^2\big)\, R \,d\theta dz \notag\\
&= -\int_{\cS_R}
	\bigl(\alpha |v_1|^2 +\gamma |v_2|^2\bigr)\, R\, d\theta dz\\
&\le -\min\{\alpha, \eps\gamma\} \| v \|^2\,. 
	\end{align*}
	In the second step, we have used that
	$\Delta_{\cS_R}$ is negative semi-definite,
	and in the third step we have
	applied the definition of the inner product 
	in Eq.~\eqref{eq:eps-inner}. 
	It follows that $\Re\, \langle (\sigma+\bar L)v, v\rangle
	\le 0$ for all $v\in H^{2,1}$.
\end{proof}

In the Fourier representation, $\bar L$ is
given by the matrix multiplication operator
\begin{equation}\label{eq:multOp}
m(k,n) := 
\begin{pmatrix}
-k^2 - n^2 R^{-2} + ick -\alpha & -1 \\
\eps & ick - \eps\gamma
\end{pmatrix}\,,\qquad (k\in\R, \,n\in \N)\,.
\end{equation}
Therefore $\bar L$ has only essential spectrum. Its resolvent
set consists of those $\lambda\in \bC$
for which $\lambda-m(k,n)$ is invertible (for all
$k\in \R$, $n\in \N$) and
$$
\sup_{k\in \R, n\in \N} \|(\lambda-m(k,n))^{-1}\|<\infty\,.
$$
The spectrum contains the branch of
eigenvalues of $m(k,0)$ given by
\begin{align*}
\lambda_+(k,0)& = ick-\frac12(k^2+\alpha+\eps\gamma)
+\frac12\sqrt{(k^2+\alpha-\eps\gamma)^2-4\eps}\\
&\sim ick-\eps\gamma\qquad (|k|\to\infty)\,.
\end{align*}
Consequently, $\bar L$ is not
sectorial, and $e^{t\bar L}$ is not an analytic semigroup.

Lemmas~\ref{lem:compact} and ~\ref{lem:Lbar} 
imply, by Lemma~\ref{lem:dissipative} and
a standard perturbation result (see \cite[Theorem 3.1.1]{Pazy}), 
that $L=\bar L +V$ 
generates a strongly continuous semigroup on $L^2$, denoted 
by $e^{tL}$. Evidently, $e^{tL}$ also fails to be analytic.
Since 
\begin{equation}
\label{eq:V-upperbound}
f'(y)-f'(0) =  -3y^2+2(\alpha+1)y 
\le 1\,,\quad \bigl(y\in\R, 0<\alpha< \tfrac12)\,,
\end{equation}
we have that $\langle Vv, v\rangle \le \|v\|^2$.
By Lemma~\ref{lem:Lbar},
\begin{align}
\label{eq:L-upperbound}
\Re\,\langle Lv,v\rangle = \Re\,\langle (\bar L + V)v,v\rangle
\le \|v\|^2\,,
\end{align}
i.e., $-1+L$ is dissipative. By Lemma~\ref{lem:dissipative},
the semigroup satisfies $\|e^{tL}\|\le e^t$.

\subsection{The spectral projection}
\label{subsec:Q}

In this subsection, we prove that
\begin{enumerate}[(i)]
	\item $\spec (L)\subset \{0\}\cup \left\{\lambda\in \bC\ \big\vert\
	\Re \, \lambda < -\sigma \right\}$ for some $\sigma>0$;
	\item $0$ is a simple eigenvalue of $L$ and $L^*$.
\end{enumerate}
This will be used to construct the spectral projection
$Q$ that appears in Proposition~\ref{prop:LQ-decay}.
We start with the essential spectrum of $L$.

\begin{lemma}[Essential spectrum]\label{lem:ess}
	The operator $L$ defined by Eq.~\eqref{eq:L} satisfies
	$$
	\spec_{ess}(L) \subset \{ \lambda\in\bC \;\vert\; \Re\, \lambda 
        \le -\sigma\}\,,
	$$
	where $\sigma=\min\{\alpha, \eps\gamma\}$.
\end{lemma}
\begin{proof}
	By Lemma~\ref{lem:compact}, $L$ is a relatively
	compact perturbation of
	the operator $\bar L$ from Eq.~\eqref{eq:Lbar}.
	It follows from Weyl's essential spectrum 
	theorem~(see \cite[Chapter~2]{KP13}) that
	$$
	\spec_{ess}(L) = \spec_{ess}(\bar L) = \spec(\bar L) \,.
	$$
	Since $\sigma+\bar L$ is dissipative
	by Lemma~\ref{lem:Lbar}, its spectrum lies
	in the left half-plane by Lemma~\ref{lem:dissipative}.
\end{proof}

To analyze the discrete spectrum of $L$,
we expand functions on $\cS_R$ as Fourier series in
the angular variable, $\theta$,
$$
v(z,\theta) = \sum_{n\in\Z} v_n(z) e^{i n\theta}\,.
$$
By Eq.~\eqref{eq:multOp}, $L$ splits into a direct sum 
$L = \bigoplus_{n\ge 0} L_n$,
where 
\begin{equation}
\label{eq:Ln}
L_n
:= \begin{pmatrix} \p_z^2-n^2R^{-2} + c\p_z + f'(\phi_1(z))
& -1 \\
\eps & c\p_z - \eps\gamma \end{pmatrix} 
\end{equation}
is the restriction of $L$ to the invariant 
subspace corresponding to the Fourier modes $\pm n$.
The next lemma concerns the positive modes.

\begin{lemma}[Resolvent estimate, $n>0$]
	\label{lem:Ln}
	Let $L_n$ be given by Eq.~\eqref{eq:Ln}.
	If $R \leq 1 $, then 
	$$\spec_{disc} (L_n)\subset \{\lambda \in \bC \;\vert\; 
	\Re\,\lambda\le -\sigma\}\, \quad (n>0)\,,
	$$
	and
$$
	\big\| (\lambda-L_n)^{-1} \big\| \leq 
	\frac{1}{Re\, \lambda+\sigma}\,, \qquad 
    (\Re \,\lambda>-\sigma,\; n>0)\,.
$$
	Here, $\sigma=\min\{\alpha,\eps\gamma\}$.
\end{lemma} 
\begin{proof} 
	Using that $f'(\phi_1(z))-f'(0) \le 1$ by 
	Eq.~\eqref{eq:V-upperbound}, we obtain  
\begin{align*}
	Re\, \langle L_n v, v \rangle &= 
	\int_\R \big((\p_{z}^2 v_1)\bar v_1 + 
	(f'(\phi_1(z)) - n^2 R^{-2})|v_1|^2- \gamma |v_2|^2\big)\,R\, dz\\
&\le -\sigma \| v \|^2\,, \qquad (n>0)\,, 
\end{align*}
	provided that $R\le 1$. (Recall that $f'(0)=-\alpha$.)
	Since $\sigma+L_n$ is dissipative,
the resolvent estimate follows from Lemma~\ref{lem:dissipative}.
\end{proof}

The heart of the matter is the discrete spectrum of the zero mode.
\begin{lemma}[Spectrum of $L_0$]
\label{lem:L0}
	Let $L_0$ be given by Eq.~\eqref{eq:Ln} with $n=0$.
	If $\eps>0$ is sufficiently small, then there
	exists $\beta>0$ 
such that
$$
	\spec_{disc}(L_0)\subset \{0\}\cup \left\{\lambda\in \bC\ \big\vert\ 
	\Re \,\lambda \le -\beta \right\}\,.
$$
	Moreover, $0$ is a simple eigenvalue of both $L_0$ and 
	its adjoint, $L_0^*$.
\end{lemma}

\begin{proof} Let $\tau=-\partial_z\Phi$ be the
	tangent vector to $\cM$ at $\Phi$.
	We have argued above that
	$L\tau=0$ by translation invariance,
	that is, 0 is an eigenvalue of $L$, 
	with eigenfunction $\tau$.
	It remains to show that the eigenvalue $0$ is simple, and 
        that there are no other eigenvalues with nonnegative
        real part.
	
	The operator $L_0$ agrees with the
	linearization of the FHN system in one
	spatial dimension.  The spectrum of this
	operator in the space of bounded continuous
	functions was analyzed by Jones~\cite{Jones_84} 
	and Yanagida~\cite{Yanagida_85}. 
	Specifically, they proved that
	for $\eps>0$ sufficiently small,
	$0$ is a simple eigenvalue,
        and all other eigenvalues lie in some half-plane 
        $\{\lambda\in \bC\ \vert\ 
	\Re\,\lambda<-\beta\}$.
	
	We claim that the discrete spectrum of $L_0$
	on $L^2(\R)$ is contained in its discrete
	spectrum on the space of
	bounded continuous functions. Indeed, any generalized 
	eigenfunction
	must lie in the domain of $L_0$, given by the mixed
	Sobolev space $H^{2,1}$, see Lemma~\ref{lem:L-domain}.
	In particular, the generalized eigenfunctions are bounded
	and continuous. Therefore the results of 
	Jones~\cite{Jones_84} and Yanagida~\cite{Yanagida_85} also
	apply to $L_0$. 
	
	Finally, by Lemma~\ref{lem:compact}, $L_0$ is a Fredholm operator
	of the same index as $\bar{L}_0$. The Fredholm index
	of $\bar{L}_0$ is zero, because 0 lies in its resolvent 
	set by Lemmas~\ref{lem:Lbar} and~\ref{lem:dissipative}. 
	Therefore, 0 is a simple eigenvalue also for $L_0^*$.
\end{proof} 

Combining Lemmas~\ref{lem:ess},~\ref{lem:Ln}, and~\ref{lem:L0}, 
we conclude that 
$$
\spec\, (L)\subset \{0\}\cup \left \{\lambda\in \bC\ \big\vert
\ \Re\,\lambda\le -\sigma\right\}\,,
$$
where $\sigma =\min \{\alpha,\beta,\eps\gamma\}$, and
$\beta$ is determined
by Lemma~\ref{lem:L0}.  Since the eigenvalue at 
zero is isolated, the Riesz projection to the zero eigenspace
is defined by the line integral
$$
P = \frac{1}{2 \pi i} \oint_{\Gamma_0} (\lambda-L)^{-1} d\lambda\,, 
$$
where $\Gamma_0 \subset \bC$ is any simple closed positively 
oriented curve that separates zero from 
the remainder of the spectrum of $L$ \cite[Chapter 2]{KP13}. 
By definition, $P$ is a bounded linear operator that
commutes with $L$ and $e^{tL}$. 

\begin{lemma} [Spectral projection]
\label{lem:P}
	Under the assumptions of Lemma~\ref{lem:L0}, 
	the Riesz projection 
	is given by 
$$
	Pu =\langle u, \tau^*\rangle \tau\,, \qquad u\in L^2\,,
$$
	where $\tau=-\p_z \phi_1$, and $\tau^*$ is the
	eigenfunction of the adjoint $L^*$ corresponding 
	to the zero eigenvalue, normalized 
	to $\langle \tau,\tau^*\rangle =1$.
\end{lemma}

\begin{proof} The eigenfunction $\tau^*$ is well-defined
	because  0 is a simple eigenvalue for both $L$ and
	$L^*$ by Lemma~\ref{lem:L0}. In particular,
	since $\tau$ does not lie in the range of $L$,
	$\tau^*$ is not orthogonal to $\tau$.
	The Riesz projection $P$ is uniquely determined
	by its action on the nullspace and range of $L$,
	i.e., the properties that $P\tau=\tau$ and $PLv=0$
	for all $v\in H^{2,1}$ hold. We verify that $\langle\tau,\tau^*\rangle\tau=\tau$,
	and $\langle Lv, \tau^*\rangle \tau = \langle v, L^*\tau^*\rangle \tau =0$,
	as required.
\end{proof}

\begin{lemma}[Resolvent estimate, $n=0$]
	\label{lem:Res0}
	Under the assumptions of Lemma~\ref{lem:L0}, 
	let $P$ be the projection to the nullspace of $L$
	constructed in Lemma~\ref{lem:P}, let 
	$Q=1-P$ be the complementary projection,
	and let $Q_0$ be its restriction to the
	$n=0$ subspace of $L^2$.
	For every $\sigma>\min\{\alpha,\beta,\eps\gamma\}$,
	there exists a positive constant $C$ such that
$$
	\| (\lambda-L_0)^{-1} Q_0\| \le C,\qquad (\Re\,\lambda \geq -\sigma)\,.
$$
Here, $\beta$ is as in Lemma~\ref{lem:L0}.
\end{lemma}

The proof of this lemma  is deferred to the next subsection.

\subsection{Proof of Lemma~\ref{lem:Res0}}

We estimate the resolvent of the operator $L_0$, 
given by Eq.~\eqref{eq:Ln} with $n=0$, separately in the three regions
\begin{align*}
	S_1 &:= \{\lambda\in\bC \;\vert\; -\sigma \leq \Re\, \lambda \leq 2,\,
|\Im\,\lambda| \leq N\}\,,\\
	S_2 &:= \{\lambda\in\bC \;\vert\; \Re\, \lambda \ge 2\}, \\
	S_3 &:= \{\lambda\in\bC \;\vert\; - \sigma 
\leq \Re\,\lambda \leq 2,\, |\Im\,\lambda| \geq N\}\,.
\end{align*}
Here, $\sigma<\min\{\alpha,\beta, \eps\gamma\}$, 
and $\beta$ is as in Lemma~\ref{lem:L0}.
The constant $N$ will be chosen in the proof of
Lemma~\ref{lem:S3}.

On $S_1$, we appeal to compactness.

\begin{lemma}[Resolvent estimate on $S_1$]
\label{lem:S1}
For any $N>0$,
$$
\sup_{\lambda\in S_1} \| (\lambda-L_0)^{-1} Q_0\|<\infty\,.
$$
\end{lemma}
\begin{proof} By Lemma~\ref{lem:L0}, $S_1$ intersects
the spectrum of $L_0$ only at the simple eigenvalue $0$.
Since the restriction of
$L_0$ to the range of $Q_0$ has no spectrum in
$S_1$, its resolvent is an analytic function of
$\lambda$, and hence bounded on the compact set~$S_1$.
\end{proof}

The half-plane $S_2$ is treated by a dissipativity estimate.

\begin{lemma}[Resolvent estimate on $S_2$]
For any $N>0$,
\label{lem:S2} 
\begin{equation}\label{eq:S2}
\sup_{\lambda \in S_2} \big\| (\lambda-L_0)^{-1} \big\| 
\leq 1\,.
\end{equation}
\end{lemma}
\begin{proof} 
Since $L_0$  is the restriction of $L$ to a subspace, 
Eq.~\eqref{eq:L-upperbound} implies that
\begin{align*}
\Re\, \langle L_0 v, v \rangle &\le \| v \|^2\,.
\end{align*}
By Lemma~\ref{lem:dissipative},
$\lambda-L_0 $ is invertible for $\Re\,\lambda > 1$ 
and the inverse satisfies
$$
\| (\lambda-L_0)^{-1} \|  
\leq \frac{1}{Re\, \lambda-1} \qquad (Re\, \lambda>1)\,.
$$
Since $\Re\,\lambda\ge 2$ on $S_2$, this 
implies Eq.~\eqref{eq:S2}. 
\end{proof}

For the region $S_3$, some explicit estimates are required.

\begin{lemma} [Resolvent estimate on $S_3$]
\label{lem:S3} If $N$ is sufficiently large then
$$
\sup_{\lambda \in S_3} \big\| (\lambda-L_0)^{-1} \big\|
\leq \frac{2}{\min\{\alpha,\eps\gamma\}-\sigma}\,.
$$
Here, $\sigma<\min\{\alpha,\beta, \eps\gamma\}$, and
$\beta$ is as in Lemma~\ref{lem:L0}.
\end{lemma}
\begin{proof} 
For $\lambda\in S_3$, we write $L_0=\bar L_0+V$, where 
$V$ is defined by this relation and
$$
\bar L_0 :=
\begin{pmatrix}
\p_z^2 + c \p_z + f'(0) & -1 \\
\eps & c\p_z - \eps\gamma
\end{pmatrix}\,.
$$
Since $\sigma + \bar L_0$ is dissipative, the resolvent set
of $\bar L_0$ contains the half-plane $\{\lambda \in \bC\ \vert\ 
\Re\,\lambda\ge -\sigma\}$.
We want to solve for $(\lambda-L_0)^{-1}$
in the resolvent identity
\begin{equation}
\label{eq:resol-id0}
(\lambda-\bar L_0)^{-1}-(\lambda-L_0)^{-1} = 
-(\lambda-\bar L_0)^{-1} V (\lambda-L_0)^{-1}\,.
\end{equation}

First, we prove that there exists an $N>0$ such that
\begin{equation}\label{eq:S3}
\sup_{\lambda\in S_3} 
\big\| (\lambda-\bar L_0)^{-1}V \big\|\le \frac12.
\end{equation}
Indeed, let $\lambda \in S_3$. Using that the $(1,2)$- and  $(2,2)$-entries of the operator matrix $(\lambda-\bar L_0)^{-1}V$ vanish and that
\begin{equation*}
\bigg\| 
\begin{pmatrix}
a_{11} & 0 \\
a_{21} & 0
\end{pmatrix}
\bigg\|
\leq \| a_{11} \| + \| a_{21} \|,
\end{equation*}
we find
\begin{equation*}
\begin{split}
\| (\lambda-\bar L_0)^{-1}V  \|
& \leq \bigl(\| ((\lambda-\bar L_0)^{-1})_{11}\|
+  \| ((\lambda-\bar L_0)^{-1})_{21}\|\bigr) \,\sup_{y\in\R}
|f'(y)-f'(0)| \,.
\end{split} 
\end{equation*}
Since the last factor is bounded by Eq.~\eqref{eq:b},
to verify Eq.~\eqref{eq:S3} it suffices to show that 
\begin{equation}\label{eq:L-infV}
\lim_{N\to\infty}
\sup_{\lambda\in S_3} \big\| ((\lambda-\bar L_0)^{-1})_{i1} 
\big\| =0\,,\quad i=1,2\,.
\end{equation}
Moreover, since the differential operator
$\bar L_0$ has real coefficients,
we may restrict the supremum to the intersection 
of $S_3$ with the upper half-plane.

In the Fourier representation, $\bar L_0$ 
becomes the matrix multiplication operator 
$$
m(k,0) = 
\begin{pmatrix}
-k^2 + ick -\alpha & -1 \\
\eps & ick - \eps\gamma
\end{pmatrix}\,,\qquad (k\in\R)\,,
$$
see Eq.~\eqref{eq:multOp}.  In particular,
\begin{equation}
\big\| ((\lambda-\bar L_0)^{-1})_{ij} \big\| = 
\sup_{k\in \R} \big|((\lambda-m(k,0))^{-1})_{ij} \big|.
\end{equation}
By Cramer's rule (suppressing the dependence on $k$ in the notation),
\begin{equation*}
((\lambda-m)^{-1})_{11} = \frac{\lambda-m_{22}}{\det(\lambda-m)}\,,\quad
((\lambda-m)^{-1})_{21} = \frac{m_{21}}{\det(\lambda-m)}\,.
\end{equation*}
Passing to reciprocals, we compute  for the first entry
\begin{align*}
\frac{1}{\bigl((\lambda-m)^{-1}\bigr)_{11}}
&= \lambda-m_{11} -\frac {m_{12}m_{21}} {\lambda-m_{22}}\,.
\end{align*}
We next separate the real and imaginary parts.
Since $\Re\, \lambda \ge -\sigma$ on $S_3$, 
and $\sigma\le\min\{\alpha,\eps\gamma\}$,
we have that
$$
\Re\, (\lambda-m_{11})\ge k^2\,,
\quad \Re\, (\lambda-m_{22})\ge \eps\gamma-\sigma>0\,.
$$
Since $m_{12}m_{21}=-\eps<0$, it follows that
\begin{equation*}
\Re\, \frac{1}{((\lambda-m)^{-1})_{11}} \ge k^2\,.
\end{equation*}
For the imaginary part, we have
$$
\Im\, \frac{1}{((\lambda-m)^{-1})_{11}}
\ge \Im\,\lambda-ck+ \eps |\Im\, \lambda-ck|^{-1}\,.
$$
Combining the two estimates yields for $\Im\, \lambda\ge N$
\begin{equation*}
\begin{split}
\frac{1}{|((\lambda-m)^{-1})_{11}|} 
&\ge \max\bigl\{k^2, |\Im \, \lambda-ck|+\eps|\Im\, \lambda-ck|^{-1}\bigr\} \\
&\ge \max\left 
\{\tfrac{N^2}{4c^2}, \tfrac{N}{2}+\tfrac{2\eps}{N}\right\} \to\infty \qquad (N\to\infty)\,.
\end{split}
\end{equation*}
The second inequality above holds since $k\ge \frac{N}{2c}$ 
whenever $|\Im \, \lambda-ck|\le \frac{N}{2}$.
This implies Eq.~\eqref{eq:L-infV} for $i=1$.
Similarly,
$$
\frac{1}{((\lambda-m)^{-1})_{21}}
= \frac{\bigl(\lambda-m_{11}\bigr) \bigl(\lambda-m_{22}\bigr)}
{m_{21}}- m_{12}\,.
$$
As before, we separately estimate the real and imaginary parts
of each of the factors in the numerator
$$
|\lambda-m_{11}| \ge \max\left\{k^2,|\Im\,\lambda-ck| \right\}\,,\quad
|\lambda-m_{22}|\ge \eps\gamma-\sigma\,.
$$
It follows that
\begin{align*}
\left| \frac{1}{((\lambda-m)^{-1})_{12}}\right|
&\ge \frac{\eps\gamma-\sigma}{\eps} 
\max \left\{k^2, |\Im \,\lambda-ck| \right\}-1\\
&\ge \frac{\eps\gamma-\sigma}{\eps}
 \max \left \{ \tfrac{N^2}{4c^2},\tfrac{N}{2} \right \} -1\\
&\quad\to\infty \quad (N\to\infty)\,.
\end{align*}
This implies Eq.~\eqref{eq:L-infV} for $i=2$.

Now choose $N$ so large that Eq.~\eqref{eq:S3} 
holds.  Since $1 - (\lambda - \bar L_0)^{-1}V$ 
is invertible, we can solve for the resolvent
of $L_0$ in Eq.~\eqref{eq:resol-id0}
to obtain 
$$
(\lambda-L_0)^{-1} = 
\big(1 - (\lambda-\bar L_0)^{-1}V \big)^{-1} (
\lambda-\bar L_0)^{-1}.
$$
Using this relation and Eq~\eqref{eq:S3}, we estimate
\begin{align*}
\| (\lambda-L_0)^{-1} \|
& \leq \bigl(\|1 - (\lambda-\bar L_0)^{-1}V\| \bigr)^{-1}
\| (\lambda-\bar L_0)^{-1} \|\\
&\le 2 \| (\lambda-\bar L_0)^{-1}\|\\
&\le \frac{2}{\Re\,\lambda+\min\{\alpha,\eps\gamma\}}\,,
\end{align*}
where the last line follows from
the dissipativity of $\min\{\alpha,\eps\gamma\}+\bar L_0$
by Lemma~\ref{lem:dissipative}.
Since $\Re\,\lambda\ge -\sigma$ on $S_3$, this proves
the claim. 
\end{proof}

\bigskip 
Lemma~\ref{lem:Res0} follows from Lemmas~\ref{lem:S1} - \ref{lem:S3}.

\subsection{Proof of Proposition~\ref{prop:LQ-decay}}
\label{subsec:LQ-decay}

Let $P$ be the projection constructed in Lemma~\ref{lem:P},
and $Q=1-P$ be the complementary projection to the range
of $L$. Choose 
$$\sigma< \min \{\alpha,\beta,\eps\gamma\}\,,
$$
where $\beta$ is the constant from Lemma~\ref{lem:L0}.
We need to find a constant $C>0$ such that
$\|e^{tL}Q\|_{2,1}\le Ce^{-\sigma t}$ for all $t>0$.

Splitting $L$ into the direct sum of its
Fourier modes, we obtain
\begin{equation*}
\big\| (\lambda-L)^{-1} Q \big\| 
= \big\| \oplus_{n\geq 0} (\lambda-L_n)^{-1} Q_n \big\| 
\le \sup_{n \ge 0} 
\left\{\big\| (\lambda-L_n)^{-1} Q_n\big\| \right\}\,,
\end{equation*}
see Eq.~\eqref{eq:Ln}.
Lemmas~\ref{lem:Ln} and~\ref{lem:Res0}
imply that there exists a constant $C>0$ such that
$$
\| (\lambda-L)^{-1} Q\| \le C\,\,,\quad (\Re\,\lambda \geq -\sigma)\,.
$$
Since $Q$ commutes with $L$, 
this establishes the hypotheses of Pr\"uss' theorem on the range of $Q$.

Applying Theorem~\ref{thm:Pruss} with $B$ equal
to the restriction of $L$ to the range of $Q$,
we obtain that 
$$\| e^{tL}Q\|\le Ce^{-\sigma t}$$
for some constant $C>0$. Since $L$ commutes with 
$e^{tL}$, it follows that
$$
\|Le^{tL}Qu\| + \|e^{tL}Qu\|\le Ce^{-\sigma t}
\bigl(\|Lu\|+\|u\|\bigr)
\qquad (u\in H^{2,1})\,.
$$
Since the $H^{2,1}$-norm is equivalent to the graph
norm of $L$ by Lemma~\ref{lem:L-domain}, Eq.~\eqref{eq:LQ-decay} follows.
This completes the proof of Proposition~\ref{prop:LQ-decay}.
\hfill $\Box$

\section{Nonlinear stability on the standard cylinder}
\label{sec:nonli-stab}

In this section, we return to the nonlinear system
on $\cS_R$ in the moving frame, 
and prove Theorem~\ref{thm:nonli-stab}.
Let $G(u)$ denote the right hand side of Eq.~\eqref{eq:FHN_moving}. 
By Proposition~\ref{prop:LWP},
the initial-value problem 
\begin{align}\label{FHN-abstr-form}
\p_t u = G(u)\,, \qquad u \vert_{t=0} = u_0\,,
\end{align}
is locally well-posed in the class of mild 
solutions on $H^{2,1}$.

\subsection{Decomposition of the solution near $\cM$}
\label{sec:split}

Let  $P$ be the projection to the tangent space of
$\cM$ at $\Phi$ from Lemma~\ref{lem:P}.
By translation invariance,
$$%
P_hv := \langle v, \tau^*_h \rangle\, \tau_h\,,\quad v\in L^2
$$
defines the corresponding projection
to the tangent space of $\cM$ at the translated
pulse $\Phi_h$. This is the spectral projection
associated with the zero eigenspace of
$L_h := dG(\Phi_h)$.

\begin{proposition}
\label{prop:split}
Under the assumptions of Proposition~\ref{prop:LQ-decay}:
\begin{enumerate}[(i)]
\item {\em (Projection onto $\cM$.)}
There exists a tubular neighborhood $\cW$ of 
$\cM$ in $H^{2,1}$ such that every
$u\in \cW$ has a unique decomposition as
\begin{equation}\label{eq:split-M}
u = \Phi_h + v \quad  \text{with} \; P_hv=0\,.
\end{equation}
\item {\em (Local projection near $\Phi$.)}
There exists a neighborhood
$\cU$ of $\Phi$ in $H^{2,1}$ 
such that every $u\in \cU$ has a unique decomposition
\begin{equation}\label{eq:split-local}
u = \Phi_h + v \quad \text{with} \; Pv=0\,.
\end{equation}
\end{enumerate}
In both cases, $h$ and $v$ are smooth functions of $u$.
\end{proposition}

In the proof, we show that
each of Eq.~\eqref{eq:split-M} and Eq.~\eqref{eq:split-local}
defines a pair of complementary non-linear projections 
$\mathcal{P}:u\mapsto \Phi_{h(u)}$ 
(onto $\cM$) and
and $\mathcal{Q}:u\mapsto v$ (onto a transversal subspace),
with $$ d\mathcal{P}\Big\vert_{u=\Phi}=P\,,
\qquad d\mathcal{Q}\Big\vert_{u=\Phi}=Q\,.
$$
In fact, Eq.~\eqref{eq:split-local}
defines a diffeomorphism $u\mapsto (h,v)$
from $\cU$ onto a neighborhood of the origin
in $\R\times {\rm Ran}\,(Q)$. The proof relies on the 
Implicit Function Theorem.
The following lemma provides the
requisite smoothness.

\begin{lemma}[Smooth dependence on $h$]
\label{lem:smooth}
The manifold of pulses $\cM$ is a smooth simple curve in $H^{2,1}$. 
Moreover, the tangent vector $\tau_h$, the
dual vector $\tau^*_h$, the projections $P_h$, $Q_h$,
the linearization $L_h=dG(\Phi_h)$,
and the nonlinearity $N_h(v):=
G(\Phi_h+v)-L_h v$
depend smoothly in $H^{2,1}$
on $h$, with bounded derivatives of all orders.
\end{lemma}

\begin{proof} The smoothness of $\cM$ follows from
the smoothness and decay of $\Phi$ and its derivatives.
This also proves the smoothness of
$\tau_h$, $\tau^*_h$, and the projections.
The linearization $L_h$ is a matrix-valued
differential operator whose coefficients are smooth functions
of $\Phi_h$; the linearity $N_h(v)=G(\Phi_h+v)-L_h(\Phi_h)v$
is a cubic polynomial
in $v_1$ whose coefficients are smooth functions
of $\Phi_h$.
Since $H^2(\cS_R)$ is a Banach algebra, 
$G(\Phi_h+v)$, $L_hv$, and $N_h(v)$ all depend smoothly on
$h$.
\end{proof}

\begin{proof}[Proof of Proposition~\ref{prop:split}]
{\em (i)}.\quad 
Given $u$ near $\cM$, we need
to find $h\in \R$ such that $P_h(u-\Phi_h)=0$. 
Choose $h_0\in\R$ with $\|u-\Phi_{h_0}\|= \dist(u,\cM)$.
By applying the translation $\tau_{-h_0}$
to a neighborhod of $u$, we may assume that $h_0=0$.
Thus we need to solve 
$$
\cH(u, h) := \langle u-\Phi_h, \tau^*_h\rangle = 0
$$
near $(\Phi,0)$.  Clearly, $\cH(\Phi,0)=0$. Moreover, since
$\Phi_h$ and $ \tau^*_h$ are smooth in $h$, the map 
$\cH$ is continuously differentiable in $h\in \R$ and $u\in L^2$.
Since $\partial_z \Phi=-\tau$, 
$$
\p_h \cH(u, h)\big|_{u=\Phi} = \langle \tau, \tau^*\rangle = 1.
$$
By the Implicit Function Theorem, there is a unique solution 
$h=h(u)$ in a neighborhood $\cU$ of $\Phi$, 
which is continuously differentiable in $u$
and satisfies $h(0)=0$.
Since $\Phi$ is smooth, also
$v(u) = u-\Phi_{h(u)}$ is smooth. The tubular neighborhood
$\cW$ is the union of all translates of $\cU$.

{\em (ii)} Apply the Implicit Function Theorem to 
$\cH(u, h) := \langle u-\Phi_h, \tau^*\rangle$.
\end{proof}

Fix a pulse $\Phi\in\cM$, and let
$\cU$ be the neighborhood constructed in 
the second part of Proposition~\ref{prop:split}. 
Consider a mild solution $u$ of Eq.~\eqref{eq:FHN_moving} 
on $\cU$.
By Eq.~\eqref{eq:split-local}, we can represent it uniquely as 
the superposition of a modulated pulse $\Phi_{h(t)}$
and a transversal fluctuation $v(t)$
\begin{equation}
\label{eq:split}
u(t)=\Phi_{h(t)}+ v(t)\quad \text{with}\; \quad Pv(t)=0\,.
\end{equation}

Since $\Phi_h$ is a stationary solution
of Eq.~\eqref{eq:FHN_static}, we have
$G(\Phi_h)=0$. Its Taylor expansion about $\Phi_h$
is given by $G(\Phi_h+v)= L_h v + N_h(v)$,
where $L_h=dG(\Phi_h)$ is as in Eq.~\eqref{eq:L}
but with $\Phi_h$ in place of $\Phi$, and
\begin{equation} \label{eq:Nh}
N_h(v) =
\begin{pmatrix}
v_1^2(\alpha +1-3(\phi_h)_1-v_1)\\ 0 
\end{pmatrix}\,.
\end{equation}
Note that $N_h$ differs from the nonlinearity $N$ in
Eq.~\eqref{eq:N} by a bounded multiplication operator that
decays as $z\to\pm \infty$.

Assume for the moment that $u$ is a classical solution of
Eq.~\eqref{eq:FHN_moving}.
Substituting Eq.~\eqref{eq:split} into Eq.~\eqref{FHN-abstr-form}
and using that $\p_t \Phi_h(z) = \dot h \tau_h$, we obtain
$$
\dot{h}(t) \tau_h  + \p_t v = G(\Phi_{h(t)} + v)
= L_h v + N_h(v)\,.
$$
We next apply the spectral projections $P$ and $Q$.  
Since $P\partial_t v=0$, we have by the chain rule
\begin{equation}
\label{eq:h}
\langle \tau_h,\tau^*\rangle \dot h =
\langle L_h v + N_h(v),\tau^*\rangle \,.
\end{equation}
The complementary projection yields
$$
\partial_t v = Q\bigl(L_hv + N_h(v)-\dot h\tau_h\bigr)\,.
$$
In general, if $u$ is a mild solution of Eq.~\eqref{eq:FHN_static}, 
we interpret $v$ as a mild solution of the equation 
\begin{equation}
\label{eq:Qv}
v(t) = e^{tL}v_0 +\int_0^t e^{(t-s)L} Q
\bigl( L_{h(s)}v(s) + N_{h(s)}(v(s))  -\dot h(s) \tau_{h(s)}\bigr)
\, ds\,.
\end{equation}

By the same argument as in Eq.~\eqref{eq:Lip-N},
the nonlinearity $N_h$ is locally Lipschitz on $H^{2,1}$.
We will need the following refined estimate
that takes advantage of the fact that
$N_h$ vanishes quadratically at $v=0$.

\begin{lemma}[Small Lipschitz estimate]
\label{lem:Nh} For any $\eta>0$
there exists a constant $C_\eta>0$ 
such that the nonlinearity $N_h(v)$ 
defined in Eq.~\eqref{eq:Nh} satisfies
\begin{equation}
\label{eq:Lipschitz}
\| N_h(v) - N_h(w) \|_{2,1}  \leq C_\eta
\max\bigl\{\|v_1\|_{H^2},\|w_1\|_{H^2} \bigr\}
\| v_1 - w_1 \|_{H^2}
\end{equation}
for all $v,w$ with $\|v_1\|_{H^2},\|w_1\|_{H^2}\le \eta$
and all $h\in\R$.
\end{lemma}

\begin{proof} We expand the first component of $N_h$ as
$$
(N_h(v))_1-(N_h(w))_1 = 
\bigl((\alpha+1-3(\phi_h)_1)(v_1+w_1)-(v_1^2+v_1w_1+w_1^2)\bigr)(v_1-w_1)\,.
$$
Eq.~\eqref{eq:Lipschitz} follows directly from the continuity
of the multiplication in $H^2$ and the fact that
$\phi_1\in H^2$.  
\end{proof}

\begin{lemma}[Evolution inequalities]
\label{lem:hv-est} 
With the notation and assumptions of Proposition~\ref{prop:LQ-decay}, 
suppose $h(t)$ and $v(t)$ satisfy 
Eqs.~\eqref{eq:h}-\eqref{eq:Qv} 
on some interval $[0,T]$, and that
$$
|h(t)|\le \kappa\,, \qquad \|v(t)\|_{2,1}\le \eta
$$ 
for all $0\le t\le T$, where $\kappa>0$ is sufficiently
small, and $\eta>0$. 
Then there exists a constant $C>0$ 
(depending on $\kappa$ and $\eta$)
such that
\begin{equation}\label{eq:h-est}
\big| \dot{h}\big| \leq C  \bigl(|h|+\|v\|_{2,1}\bigr)\|v\|_{2,1}\,,
\qquad (0\le t\le T)\,,
\end{equation}
and
\begin{equation}
\label{eq:v-est}
\|v(t)\|_{2,1}\le 
C_0e^{-\sigma t}\|v_0\|_{2,1} \!+ C\int_0^t e^{ -\sigma (t-s)} 
\bigl(|h(t)|+\|v(t)\|_{2,1}\bigr)\|v(t) \|_{2,1}\, ds\,.
\end{equation}
\end{lemma}
\begin{proof} 
Choose $\kappa>0$ such that $\langle \tau_h,\tau^*\rangle\ge \frac12$ 
for all $h$ with $|h|\le \kappa$.
This is possible because $\langle \tau,\tau^*\rangle=1$, 
and $\langle\tau_h,\tau^*\rangle$ depends
smoothly on $h$ by Lemma~\ref{lem:smooth}. 
For $|h|\le \kappa$, Eq.~\eqref{eq:h} yields
$$
|\dot h|\le 2\bigl(|\langle L_h v, \tau^*\rangle|+ 
|\langle N_h(v),\tau^*\rangle|\bigr)\,.
$$
Since $L^*v=0$, the first summand is bounded by
$$
|\langle L_h v, \tau^*\rangle|
= |\langle v, (L^*_h-L^*)\tau^*\rangle |
\le C|h|\, \|v\|_{2,1}
$$
for some constant $C$.
By Lemma~\ref{lem:Nh}, the second summand satisfies
$\| N_h(v) \|_{2,1} \leq C_\eta \|v\|_{2,1}^2$.
Combining these two inequalities yields the bound on $|\dot h|$.

For $v(t)$, we separately estimate each term 
on the right hand side of Eq.~\eqref{eq:Qv}.
In the integrand, we use that $Lv=0$ and
$\|L_h-L\|_{2,1}$ is of order $|h|$ by Lemma~\ref{lem:smooth}.
For the nonlinearity, we use
Lemma~\ref{lem:Nh}, and for the last term, we use
Eq.~\eqref{eq:h-est}. The result is
\begin{align*}
\|v(t)\|_{2,1}
\le \|e^{tL}Q\|_{2,1}\|v_0\|_{2,1}+ C \int_0^t \| e^{-(t-s)L} Q\|_{2,1}
\bigl(|h(s)| + \|v(s)\|_{2,1}\bigr)\|v(s)\|_{2,1}\, ds\,.
\end{align*}
The proof is completed with Proposition~\ref{prop:LQ-decay}.
\end{proof}

We end this subsection with a differential inequality
that will be used below.

\begin{lemma}[Estimates on $h$]
\label{lem:h-ineq}
Let $y$ be a nonnegative, nondecreasing 
function on $[0,t]$, and let $C,\xi, \eta$ be nonnegative
constants.  If $h$ satisfies the differential inequality
$$
|\dot h| \le Ce^{-\xi t}y(|h|+y)\,
\qquad (0\le t\le T)
$$
with $h(0)=0$, and $y(T)\le\eta$, then
$$|h(t)| \le C_1 y(t)\,,
\quad |\dot h(t)|\le C_2 
e^{-\xi t} y(t)^2\, \qquad (0\le t\le T)\,,
$$
where $C_1=(e^{\frac {C\eta}{\xi}}-1)$, 
$C_2= Ce^{\frac {C\eta}{\xi}}$.
\end{lemma}

\begin{proof} Fix $t_0\in (0, T]$. 
Since $y$ is non-decreasing, 
$|\dot h(s)| \le C y(t)(|h(t)|+y(t))$ for all $0\le s\le t$.
We separate variables and integrate from
$h(0)=0$ to obtain
$$
\log \Bigl( \frac{|h(s)|+y(t)}{y(t)}\Bigr)
\le \frac{C y(t)}{\xi}(1-e^{-\xi t})\le \frac{C\eta}{\xi}\,
\qquad (0\le s\le t)\,,
$$
which yields the first claim after solving for $h(s)$ and
setting $s=t$.
The second claim follows by substituting this
bound back into the differential inequality.
\end{proof}

\subsection{Proof of Theorem~\ref{thm:nonli-stab}}
\label{sec:stab-proof}

Let $L$ be the linearization about $\Phi$ in the
moving frame, defined in Eq.~\eqref{eq:L},
and $\xi\in (0,\sigma)$.
We will construct a neighborhood $\cU$ of
$\Phi$ in $H^{2,1}$ such that for every solution $u(t)$
with initial value $u_0\in \cU$ there is a real-valued function
$h(t)$ such that
$$
\|u(t)-\Phi_{h(t)}\| \lesssim e^{-\xi t}\|u_0-\Phi\|_{2,1}\,
\qquad (t\ge 0)
$$
where $ |h(0)|\lesssim \|u_0-\Phi\|_{2,1}$, and there exists
$h_*\in\R$ such that
$$
|h(t)-h_*| \lesssim e^{-\xi t}\|u_0-\Phi\|_{2,1}^2
\qquad (t\ge 0)\,.
$$
Transforming back to the static frame, this will prove the theorem.

By Eq.~\eqref{eq:split-M} of Proposition~\ref{prop:split},
there is a neighborhood $\cU$ of $\Phi$ such that
each $u_0\in\cU$ can be written uniquely
as $u_0=\Phi_{h_0} + v_0$, 
where $P_{h_0}v=0$.  Since $h_0$ depends smoothly on $u$, 
$$|h_0|\lesssim \|u_0-\Phi\|_{2,1}\,,
\quad \|v_0\|_{2,1}= \|u_0-\Phi_{h_0}\|_{2,1}
\lesssim \|u_0-\Phi\|_{2,1}\,.
$$
Replacing $u_0$ with its translate $(u_0)_{-h_0}$, 
we may assume that $h_0=0$, that is, 
$$
u_0=\Phi+v_0\,,\quad Pv_0=0\,.
$$
By the second part of Proposition~\ref{prop:split}, 
the solution of Eq.~\eqref{eq:FHN_moving}
with initial value $u_0$ can be written uniquely as 
$$
u(t)= \Phi_{h(t)} + v(t)\,,\qquad Pv(t)=0\,,
$$
so long as $u(t)\in \cU$.
The functions $h(t)$ and $v(t)$ satisfy
inequalities~\eqref{eq:h} and~\eqref{eq:Qv}
with initial values $h(0)=0$ and $v(0)=v_0$. 

Since the map $u\mapsto (h,v)$ is a diffeomorphism
from $\cU$ to a neighborhood of the origin
in $\R\times \Ran (Q)$, by replacing $\cU$
with a smaller neighborhood 
we may assume that it has the form 
$\cU=\left\{\Phi_h+v\ \big\vert \ 
(h,v)\in \R\times \Ran(Q), 
|h|<\kappa, \|v\|_{2,1}<\eta\right\}$,
where $\kappa$ is so small that $\langle \tau_h,\tau^*\rangle\ge 
\frac12$ whenever $|h|\le \kappa$. The value of $\eta>0$ will
be further specified below. 

Let $\sigma$ be the exponent from 
Proposition~\ref{prop:LQ-decay}, and let $C_0$ be
the multiplicative constant.
Choose $\xi\in (0,\sigma)$, 
define the monotonically increasing function
$$
y(t):=\sup_{0\le s\le t} e^{-\xi s}\|v(s)\|_{2,1}\,,
$$
and let 
$$
T:=\inf\left\{t\ge 0\ \big\vert\ 
|h(s)|\ge \kappa\ \text{or}\ \|y(s)\|_{2,1}\ge\eta \right\}\,.
$$ 

Assume that $\|v_0\|_{2,1}<\frac{\eta}{2C}$, and
apply Lemma~\ref{lem:hv-est}.  By Eq.~\eqref{eq:h-est},
$$
|\dot h(t)|\le C e^{-\xi t}\bigl(|h(t)| +y(t)\bigr)y(t)\,,\qquad
(0\le t\le T)\,,
$$
where we have used that $\|v(t)\|_{2,1}\le
e^{-\xi t} y(t)$ by definition of $y$.
It follows by Lemma~\ref{lem:h-ineq} that
$ h(t) \le C_\eta y(t)$ for $0\le t\le T$
for some constant $C_\eta$. Since $y(t)<\eta$
for $t<T$, by reducing the value of $\eta$ 
we can achieve that $|h(t)|<\kappa$ for all $t\in [0,T)$.
Inserting this estimate into Eq.~\eqref{eq:v-est} yields
$$
y(t)\leq C_0 e^{-(\sigma-\xi)t}\| v_0 \|_{2,1} 
+ C \int_0^t e^{-(\sigma-\xi)(t-s)}
y^2(s)\, ds\,,
$$
where $C$ is the product of $C_0$, $C_\eta$, and
the constant from Lemma~\ref{lem:Nh}.
Since $y$ is nondecreasing, taking it out of the integral
yields the upper bound
\begin{align}\label{eq:y}
y(t) & \leq C_0\| v_0 \|_{2,1} + C y^2(t)\, 
\qquad (0\leq t \leq T)
\end{align}
with a suitably adjusted constant $C$.

Consider the quadratic polynomial
$p(y):=C_0\|v_0\|_{2,1}-y+ C y^2$. If
$d:=4C_0C \|v_0\|_{2,1}<1$, then $P$
has two positive real roots, and is positive
on the interval between them.
The smaller root satisfies
$$
y_*=\frac{1}{2C}
\Bigl(1-\sqrt{1-4C_0C\|v_0\|_{2,1}}\Bigr)
\le 2C\|v_0\|_{2,1}< \eta \,.
$$
Since $C_0\ge 1$, we have that $\|v_0\|_{2,1}<y_*$.
Eq.~\eqref{eq:y} implies, by continuity,
that $y(t)\le y_*$ for all $0\le t\le T$.
If $T<\infty$, then by continuity also $y(T)\le y_*<\eta$,
contradicting the definition of $T$.
Hence $T=+\infty$, and
$$
\|v(t)\|_{2,1}\le e^{-\xi t}y(t) \le 2C_0e^{-\xi t}\|v_0\|_{2,1}\,
\qquad (t\ge 0)\,.
$$
Since $|h(t)|\le \kappa$ and $\|v(t)\|<\eta$, we conclude that
$\Phi_{h(t)}+v(t)\in \cU$ for all $t\ge 0$,
and $\|v(t)\|_{2,1}$ converges exponentially to zero. 

To show that $h(t)$ converges as well,
we use again Lemma~\ref{lem:h-ineq} to see that
$$
|\dot h(t)| \le C e^{-\xi t}\|v_0\|_{2,1}^2\,.
$$
It follows that
$h(t)$ converges exponentially to a limit, $h_*$,
with $|h_*|\lesssim \|v_0\|_{2,1}^2$. 
Since $\|v_0\|_{2,1}\lesssim \|u_0-\Phi\|_{2,1}$,
this proves the estimate for $h$.
The proof of the theorem is completed by shrinking
the neighborhood once more, to
$$\cU=\left\{\Phi_h+v\ \big\vert \ 
(h,v)\in \R\times \Ran(Q), 
|h|<\kappa, \|v\|_{2,1}<\tfrac{\eta}{2C}\right\}\,\\[-0.7cm]
$$
\hfill $\Box$ 

\vspace{0.2cm}

\section{Near-pulse solutions on warped cylinders}
\label{sec:warped}

This section is dedicated to the proof of Theorem~\ref{thm:persist}.
Consider the FHNcyl system~\eqref{eq:FHN} 
on a warped cylinder $\cS_{\rho}$, given by Eq.~\eqref{eq:FHN_static}.

In the special case where $\rho\equiv R$, Eq.~\eqref{eq:FHN_static} 
equivalent to Eq.~\eqref{eq:FHN_moving}, 
expressed in the static frame.
The pulse defines a traveling wave solution  $\Phi(x-ct)$
on $\cS_R$.  As discussed in the introduction, the proof of
Theorem~\ref{thm:persist} relies on a perturbation 
estimate that controls the dependence of solutions 
on $\rho$.  The size of the perturbation
is measured in terms of the
essential parameter $\delta:=R^{-1}\|\rho-R\|_{C^2}$.

\begin{proposition} [Perturbation of the radius]
\label{prop:continuous-rho} 
Suppose that $u\in C([0,T], H^{2,1})$
is a mild solution of Eq.~\eqref{eq:FHN_static} on a
standard cylinder $\cS_R$,
with initial value $u_0:=u\big\vert_{t=0}$.
There are positive constants $\delta_*$ and $C$
(which depend on $T$ and on $\sup_{0\le t\le T}\|u(t)\|_{2,1}$)
such that if $0<R\le 1$
and $\delta:= R^{-1}\|\rho-R\|_{C^2}\le \delta_*$, then
the unique mild solution of Eq.~\eqref{eq:FHN_static}
on $\cS_\rho$ with
initial values $u_\rho\big\vert_{t=0}=u_0$ satisfies
$$
\sup_{0\le t\le T} \|u_\rho(t)-u(t)\|_{2,1}
\le C\delta\,.
$$
\end{proposition}

The Riemannian structure
on $\cS_\rho$ induces an alternative inner product on $L^2(\cS_\rho)$,
\begin{equation}\label{eq:inner-rho}
\langle u,w \rangle_\rho :=\int_{\cS_{\rho}} 
(u_1\bar{w}_1+\eps^{-1}u_2\bar{w}_2)\, d\mu_\rho\,,
\end{equation}
where $d\mu_\rho = \sqrt{g}\, d\theta dx$ is the Riemannian
area element whose density is determined by $g= \rho^2(1+\rho'^2)$.
The corresponding norm will be denoted by $\| \cdot \|_\rho$. 
We also define the mixed Sobolev spaces
\begin{align}\label{eq:Hkl}
H^{2k,\ell}(\cS_\rho)
:=\left\{u\in L^2
\ \big\vert\  
(\Delta_{\cS_\rho})^k u_1 \in L^2 (\cS_\rho),
  (\p_x)^\ell u_2\in L^2(\cS_\rho) \right\}
\end{align}
for $k,\ell=0,1$,
with norms 
$$\|u\|_{2k,\ell;\rho}:=\sum_{0\le i\le k}
\|(\Delta_{\cS_\rho})^iu_1\|_\rho + 
\eps^{-1}\sum_{0\le j\le \ell}
\|\partial_x^j u_2\|_\rho\,.
$$
For $k=1, \ell=0$, the space $H^{2, 0}$ agrees with 
the corresponding Sobolev space $H^{2}\times L^2$,
and $H^{0,0}=L^2$. However,
for $\ell=1$, since $H^{2k,1}$
places no condition on $\partial_\theta u_2$,
the space $H^{0,1}$ properly contains
$L^2\times H^1$, and $H^{2,1}$ properly contains
$H^{2}\times H^1$. 
On the standard cylinder~$\cS_R$,
Eq.~\eqref{eq:Hkl} with $k=\ell=1$ coincides
with the definition of $H^{2,1}$ in Eq.~\eqref{eq:H21}.

We will show in Lemma~\ref{lem:21-equivalent}
that the norms $\|\cdot\|_{2k,\ell;\rho}$ are 
equivalent to $\|\cdot\|_{2k,\ell}$.  
Hence the cylindrical surface $\cS_\rho$ 
will be omitted from the notation 
whenever this is possible without causing confusion.

\subsection{The linear semigroup}
\label{subsec:Lin-A}

The linearization of Eq.~\eqref{eq:FHN_static}
about zero is given by the G\^ateaux derivative 
$A_\rho=dF_{\rho}(0)$.
We start with some basic properties of 
$A_\rho$. Throughout this subsection,
$\rho$ is fixed subject to the standing assumption.
The domain of $A_\rho$ (as an operator
on $L^2(\cS_\rho)$) is $H^{2,0}(\cS_\rho)$, and
its graph norm is equivalent to $\|\cdot\|_{2,0;\rho}$.
Since 
$$
\Re\,\langle A_\rho u, u\rangle_\rho  \le -\sigma\|u\|_\rho^2\,,
\quad u\in H^{2,1}(\cS_\rho)\,,
$$ 
where $\sigma=\{\alpha,\eps\gamma\}$, the graph
norm of $A_\rho$ is equivalent to $\|A_\rho u\|_\rho$,
\begin{equation}
\label{eq:domain-A}
\|u\|_{2,0;\rho}\lesssim \|A_\rho u\|_\rho
\lesssim \|u\|_{2,0;\rho}\,.
\end{equation}
In the same way as for the operator $\bar L$
in Lemma~\ref{lem:Lbar}, it follows with
Lemma~\ref{lem:dissipative} 
that $A_\rho$ generates a strongly continuous, exponentially 
decaying semigroup on $L^2(\cS_\rho)$ that has $H^{2,0}(\cS_\rho)$
as an invariant subspace.  

We want to work in the subspace $H^{2,1}$ that was used 
for Theorem~\ref{thm:nonli-stab}. To this end, 
we first restrict $A_\rho$ to the intermediate 
subspace $H^{0,1}(\cS_\rho)$.

\begin{lemma}[Domain of $A_\rho$ in $H^{0,1}$] 
\label{lem:domain-A}
Let $\alpha$, $\gamma$, $\eps$ be fixed positive constants,
and let $\rho$ be a positive function of class $C^2$
on the real line.  Then the 
operator $A_\rho$ maps $H^{2,1}(\cS_\rho)$
bijectively onto $H^{0,1}(\cS_\rho)$, and
\begin{equation}\label{eq:domain-A-01}
\| u \|_{2,1;\rho} 
\lesssim \| A_\rho u \|_{0,1;\rho}
\lesssim \| u \|_{2,1;\rho} \qquad (u \in H^{2,1}(\cS_\rho))\,.
	\end{equation}
\end{lemma}
\begin{proof} 
Fix $\rho$ as in the assumptions, and let $u\in H^{2,1}(\cS_\rho)$. 
To simplify notation, we momentarily
suppress the dependence of the spaces and norms on $\rho$ 
in the notation.

Write $\|A_\rho u\|_{0,1}
= \|A_\rho u\| + \|\partial_x(\eps u_1-\eps\gamma u_2)\|$, 
and combine the upper bound in Eq.~\eqref{eq:domain-A}
with the estimates $\|\partial_x u_1\|\le \|u\|_{2,1}$
and $\|\partial_xu_2\|\le \|u\|_{0,1}$.
It follows that $A_\rho u\in H^{0,1}$, 
and the upper bound in Eq.~\eqref{eq:domain-A-01}
holds. In particular, $A_\rho$ maps $H^{2,1}$ to $H^{0,1}$.
By  Eq.~\eqref{eq:domain-A},
this map is injective.

To show that this map is also surjective, let $w\in H^{0,1}$.
Since $w\in L^2$,
the equation $A_\rho u=w$ has a unique solution
$u\in H^{2,0}$. The lower bound in Eq.~\eqref{eq:domain-A}
yields $u_1\in H^2$, and $\|u_1\|_{H^2}
\lesssim \|A_\rho u\|\le \|A_\rho u\|_{0,1}$.
For the second component, we use that
$\eps u_1-\eps\gamma u_2= w_2$, and estimate
$$
\|\partial_x u_2\|\le \gamma^{-1} \|\partial_x u_1\| +
(\eps\gamma)^{-1} \|\partial_x w_2 \| 
\lesssim \|u\|_{2,0} + \|w\|_{0,1} \lesssim \|A_\rho u\|_{0,1}\,.
$$
This proves surjectivity, and the lower bound.
\end{proof}

A useful consequence of Lemma~\ref{lem:domain-A} is that
$$\|B_{0,1;\rho}\|\lesssim\|B\|_{2,1;\rho}\lesssim \|B\|_{0,1;\rho}$$
for every bounded linear operator $B$
on $H^{0,1}(\cS_\rho)$ that commutes with $A_\rho$.
The next lemma provides spectral estimates
on $A_\rho$ that are needed to construct
the semigroup $e^{tA_\rho}$ on $H^{0,1}(\cS_\rho)$.

\begin{lemma}[$A_\rho$ is sectorial]
\label{lem:sectorial} Let $\alpha$, $\gamma$, $\eps$ be positive
constants, and let $\rho$ be a real-valued
function on $\R$ that is bounded and bounded away from zero.
Then $A_\rho$ generates an
analytic semigroup $e^{tA_\rho}$ on $H^{2,1}$.
The spectrum of $A_\rho$ on $H^{0,1}(\cS_\rho)$ is contained in 
the truncated sector
$$
	\Sigma:= \bigl\{ \lambda\in \bC \ \big\vert \ 
\Re\,\lambda \le 
-\sigma \min\{1, \eps^{-\frac12}\, | \Im\, \lambda|\} \bigr\} \,,
$$
where $\sigma:=\min\{\alpha,\eps\gamma\}$.
Moreover, we have the resolvent estimate
\begin{equation}
\label{eq:sectorial}
\|(\lambda-A)^{-1}\|_{2k,1;\rho}\le 
C(1+\sup|\rho'|)\min\left\{1,|\lambda|^{-1}\right\}\,,\qquad (k=0,1)
\end{equation}
for all $\lambda$ with
$\Re\,\lambda \ge -\frac12 \sigma 
\min\{1, \eps^{-\frac12}\, | \Im\, \lambda|\}$,
for some constant $C$. 
The semigroup satisfies
$$
\sup_{t>0}\|e^{tA_\rho}\|_{2,1}\le C
(1+\sup|\rho'|)e^{-\sigma t}\,.
$$
\end{lemma}

\begin{proof} We will bound the numerical range
of $A_{\rho}$ with respect
to a certain weighted inner product, and then 
apply~\cite[Theorem 1.3.9]{Pazy}.
As in the proof of Lemma~\ref{lem:domain-A},
$\rho$ is fixed and will be suppressed in the notation.
For $s> 0$, define 
$$
\langle u, v \rangle_s :=
\langle u_1, v_1 \rangle 
+ \eps^{-1}\langle u_2, v_2 \rangle + s\eps^{-1}
\langle \partial_x u_2, \partial_x v_2\rangle\,.
$$
The corresponding norm $\|\cdot\|_s$
is equivalent to the norm $\|\cdot\|_{0,1}$
from Eq.~\eqref{eq:Hkl},
$$ \min\{1,s^\frac12\} \|u\|_{0,1}\le \|u\|_s
\le \max\{1,s^{\frac12}\}\|u\|_{0,1}\,,\qquad (s>0)\,.
$$
We compute
\begin{align}
\notag \Re\, 
\langle A_\rho u, u \rangle_s
&= -\alpha \|u_1\|^2
-\gamma \|u_2\|^2  +
\langle \Delta_{\cS_\rho} u_1, u_1\rangle + s\Re\,  
\langle \partial_x u_1-\gamma\partial_xu_2, \partial_x u_2\rangle
\\
\notag
&\hskip -1cm\le -\sigma\|u\|^2 
-\frac{1}{1+\sup|\rho'|^2}\|\partial_x u_1\|^2
+s \|\partial_x u_1\|\|\partial_x u_2\|
-s\gamma\|\partial_x u_2\|^2\\
\label{eq:sectorial-Re}
&\hskip -1cm
\le -\sigma\|u\|_s^2 +\frac{s^2}4 (1+\sup|\rho'|^2)\|\partial_x u_2\|^2\,.
\end{align}
Note that the inner products and norms on the right
hand side of the first line are the standard Riemannian
ones for scalar functions in $L^2(\cS_\rho)$.
In the second line, we have integrated
the Laplacian term by parts.
The last step follows by completing the square.
For any $q\in (0,1)$, we can achieve
$\Re\, \langle A_\rho u,u\rangle_s\le -q\sigma\|u\|_s^2$
by choosing $s$ sufficiently small.
By Lemma~\ref{lem:dissipative},
$A_\rho$ generates a semigroup of contractions
with respect to the norm $\|\cdot\|_s$.
Moreover, the spectrum of $A_\rho$ is contained in
each of the half-planes $\{\Re\, \lambda\le-q\sigma\}$, and hence 
in their intersection. 

Likewise,
\begin{align*}
\Im\, \langle A_\rho u,u\rangle_s 
& = 2 \Im\, \langle u_1, u_2\rangle
+ s\Im\, \langle \partial_x u_1,\partial_x u_2\rangle\\
&\le \sqrt{\eps}\|u\|^2 + s\|\partial_x u_1\|
\|\partial_xu_2\|\,.
\end{align*}
Comparing with the second line of Eq.~\eqref{eq:sectorial-Re},
we see that for $s>0$ sufficiently small
$$\Re\,\langle A_\rho u, u\rangle_s \le -\sigma\eps^{-\frac12}
|\Im\, \langle A_\rho u, u\rangle_s|\,.  
$$
Since the resolvent set of $A_\rho$ contains 0,
by \cite[Theorem~1.3.9]{Pazy},
it contains the entire complement of the sector $\{\Re\,\lambda\le 
-\sigma \eps^{-\frac12}|\Im\, \lambda|\}$.
In summary, for $s$ sufficiently small, the
numerical range of $A_\rho$ with 
respect to $\langle\cdot,\cdot\rangle_s$
lies in 
$$\Sigma_q:=\left\{\lambda\in\bC\ \big\vert\ \Re\,\lambda
\le -\sigma\min\{q, \eps^{-\frac12}|\Im\, \lambda|\}\right\}\,.
$$
Moreover, the resolvent satisfies
$$
\|(\lambda-A)^{-1}\|_s\le \left(\inf_{z\in \Sigma_q}
\|\lambda-z\|_s\right)^{-1}\,,\qquad (\lambda\not\in \Sigma_q)\,.
$$
To obtain Eq.~\eqref{eq:sectorial}, we choose $q=\frac34$,
$s= \gamma(1+\sup|\rho'|^2)^{-1}$,
and compare
$\|\cdot\|_s$ with $\|\cdot\|_{0,1}$.  Since the resolvent commutes
with $A_\rho$, by Lemma~\ref{lem:domain-A},
the estimate holds, with a suitably adjusted
constant, also for the norm $\|\cdot\|_{2,1}$.
By Lemma~\ref{lem:dissipative}, the
bound on the semigroup follows from the
dissipativity of $A_\rho$.
\end{proof}

\subsection{Comparison with the standard cylinder}
\label{subsec:compare}

In this subsection we compare solutions of the FHNcyl 
system on $\cS_\rho$ with solutions on the standard
cylinder $\cS_R$. For this purpose, we consider 
functions $u(x,\theta)$ on $\cS_\rho$ as
functions on $\cS_R$, via the 
coordinate diffeomorphism $\psi_\rho:\cS_R\to\cS_\rho$
defined in Eq.~\eqref{eq:diffeo}.
We next show that the composition $u\mapsto u\circ\psi_\rho$ induces
a bounded linear transformation
from $H^{2,1}(\cS_\rho)$ to $H^{2,1}(\cS_{R})$.
With a slight abuse of notation, we identify the
composition $u\circ\psi_\rho$ (viewed as a function on the standard
cylinder $\cS_R$) with $u$ itself (on $\cS_\rho$),
and the norm $\|\cdot\|_{2,1;\rho}$ (originally
defined on $H^{2,1}(\cS_\rho)$)
with its pull-back to $H^{2,1}(\cS_R)$.

\begin{lemma} [Equivalence of Sobolev spaces]
\label{lem:21-equivalent} 
Let $\alpha$, $\gamma$, $\eps$ be fixed positive 
constants, and let $\rho$ be a positive function of class $C^2$.
Under the assumptions of Proposition~\ref{prop:difference},
if $\delta:=R^{-1}\|\rho-R\|_{C^2}\le \frac{1}{16}$, then
$$
\frac12 \|u\|_{2k,\ell}\le \|u\|_{2k,\ell;\rho} \le 2\|u\|_{2k,\ell}\,
\qquad (k,\ell\in\{0,1\})\,.
$$
\end{lemma}

\begin{proof} 
We will show that for every scalar-valued function $w$ on $\cS_\rho$,
\begin{equation}
\label{eq:rho-R}
\begin{cases}
\left| \|w\|_{L^2(\cS_\rho)}- \|w\|_{L^2(\cS_R)}\right|
\le 2\delta \|w\|_{L^2(\cS_R)}\,,\\
\left| \|\partial_x w\|_{L^2(\cS_\rho)}- \|\partial_x w\|_{L^2(\cS_R)}\right|
\le 2\delta  \|\partial_x w\|_{L^2(\cS_R)}\,,\\
\left| \|\Delta_{\cS_\rho} w\|_{L^2(\cS_\rho)}- 
\|\Delta_{\cS_R}w\|_{L^2(\cS_R)}\right|
\le 8\delta 
\left( \|\Delta_{\cS_R}w\|_{L^2(\cS_R)}+ \|w\|_{L^2(\cS_R)}\right)\,,
\end{cases}
\end{equation}
and then apply the triangle inequality.

For the first line, we have by Eq.~\eqref{eq:inner-rho}
$$
\|w\|_\rho^2 =\int_\R\int_{S^1} |w|^2 \, 
\sqrt{g(x)} \, d\theta dx\,,
$$
where $g=\rho^2(1+(\rho')^2)$.  The pointwise bound
$|\sqrt{g(x)}-R|\le 2\delta$ yields
the first line of Eq.~\eqref{eq:rho-R}. 
The second line follows by applying the first
one to $\partial_x w$.

For the third line, we write the difference between the 
Laplacians as
$$
\Delta_{\cS_\rho}-\Delta_{\cS_R}
= a(x) \partial_x^2 + b(x) \partial_x
+ c(x) \, |\rho(x)-R|\, (R^{-2}\, \partial_\theta^2)\,,
$$
where the coefficients are pointwise bounded by
$0<a(x)\le \frac12 R\delta$, $|b(x)|\le \delta+ R^2\delta^2$,
and $c(x)\le 6 R\delta$, see Eq.~\eqref{eq:Delta-rho}.
Since $\delta\le1$, $R\le 1$, and
$\|\partial_x w\|\le \frac12(\|\Delta_{\cS_R} w\| +\|w\|)$,
we obtain with the triangle inequality that
\begin{equation}
\label{eq:rho-R-4}
\|(\Delta_{\cS_\rho}-\Delta_{\cS_R})w\|
\le  6\delta \left( \|\Delta_{\cS_R}w\| + \|w\|\right)\,.
\end{equation}
Using once more the triangle inequality, 
as well as the first line of Eq.~\eqref{eq:rho-R},
we arrive at
\begin{align*}
\bigl| \|\Delta_{\cS_\rho}w\|_\rho - \|\Delta_{\cS_R}w\| \bigr|
&\le
\|(\Delta_{\cS_\rho}-\Delta_{\cS_R})w\| + 
\bigl| \| \Delta_{\cS_\rho}w\|_{\rho}-\| \Delta_{\cS_\rho}w\|
\bigr|\\
&\le 8\delta \left( \|\Delta_{\cS_R}w\| + \|w\|\right)\,.
\end{align*}
When $\delta\le \frac1{16}$, we can 
solve Eq.~\eqref{eq:rho-R} for 
the norm on $L^2(\cS_\rho)$ to obtain
$$
\frac12 \|w\|_{L^2(\cS_R)}\le \|w\|_{L^2(\cS_\rho)}
\le  2\|w\|_{L^2(\cS_R)}\,,
$$
and likewise for $\|\partial_x w\|_\rho$ and $\|\Delta_{\cS_\rho}w\|_\rho$.
By the definition of the norms in Eqs.~\eqref{eq:H21} and~\eqref{eq:Hkl},
this proves the claim.
\end{proof}

From now on, we identify the spaces
$H^{2k,\ell}(\cS_\rho)$ with $H^{2k,\ell}(\cS_R)$,
and use the standard norms $\|\cdot\|_{2k,\ell}$.
Under the assumptions of Lemma~\ref{lem:21-equivalent},
the inequalities in
Lemmas~\ref{lem:domain-A} and~\ref{lem:sectorial}
hold also for these norms.

The next lemma bounds the difference
between the resolvents of $A_\rho$ and $A_R$.

\begin{lemma} [Perturbation estimate for the resolvent]
\label{lem:diff-res}
Let $\alpha$, $\gamma$ and $\eps$ be positive constants,
and $0<R\le 1$.
There exists a constant $C$ 
such that if $\delta:=R^{-1}\|\rho-R\|_{C^2}\le \frac{1}{16}$, then
$$
\| (\lambda - A_\rho)^{-1} - (\lambda - A_R)^{-1} \|_{2,1}
	\le C\delta \min\{1, |\lambda|^{-1}\}
$$
for all $\lambda$ with $\Re\,\lambda\ge -\frac12 \sigma\min\{1,
\eps^{-\frac12}|\Im\, \lambda|\}$.
\end{lemma}

\begin{proof}
If $\lambda$ is as in the statement of the
lemma, then by Lemma~\ref{lem:sectorial}
it lies in the resolvent set of both $A_\rho$ and $A_R$.
To estimate the difference, we write
$$
A_\rho-A_R= \begin{pmatrix}
\Delta_{\cS_\rho} -\Delta_{\cS_R} & 0 \\
0 & 0
\end{pmatrix}=:W\,,
$$
and apply the resolvent identity 
$$(\lambda\!-\! A_\rho)^{-1} - (\lambda\!-\! A_{R})^{-1} =
(\lambda\!-\! A_R)^{-1} W(\lambda\!-\!A_\rho)^{-1}\,.
$$
For the factor on the right, we use Lemmas~\ref{lem:21-equivalent}
and~\ref{lem:domain-A} to see that
$$
\|(\lambda-A_\rho)^{-1}\|_{2,1}
\lesssim \|(\lambda-A_\rho)^{-1}\|_{2,1;\rho}
\lesssim \|(\lambda-A_\rho)^{-1}\|_{0,1;\rho}
\lesssim \min\{1, |\lambda|^{-1}\}\,.
$$
The second inequality holds because the resolvent commutes with
$A_\rho$.
By Eq.~\eqref{eq:rho-R-4}, the middle factor
maps $H^{2,1}$ into $H^{0,1}$ and satisfies
$$
\|W u\|_{0,1}= \|(\Delta_{\cS_\rho}-\Delta_{\cS_R})u_1\|
\lesssim \delta\|u\|_{2,1}
$$
for all $u\in H^{2,1}$.
By Lemma~\ref{lem:domain-A}, the factor on the left maps
$H^{0,1}$ back into $H^{2,1}$, and
\begin{align*}
\|(\lambda-A_R)^{-1}u\|_{2,1}
&\lesssim \|A_R(\lambda-A_R)^{-1}u\|_{0,1} \\
&\le \|u\|_{0,1} + |\lambda|\, \|(\lambda-A_R)^{-1}u\|_{0,1} \\
&\lesssim \|u\|_{0,1}
\end{align*}
for all $u\in H^{0,1}$.
In the second line we have written $A_R=-(\lambda-A_R)+\lambda$
and applied the triangle inequality, and in the last line
we have used that $\|(\lambda-A_R)^{-1}\|_{0,1}
\lesssim \min\{1, |\lambda|^{-1}\}$ by Lemma~\ref{lem:sectorial}.

Combining the inequalities for the three factors, we conclude that
$$
\left\|\bigl((\lambda\!-\!A_\rho)^{-1}-(\lambda\!-\!A_R)^{-1}\bigr)u
\right\|_{2,1} 
\lesssim \delta \min\{1, |\lambda|^{-1}|\}\, \|u\|_{2,1}
$$
for all $u\in H^{2,1}$,
proving the claim.
\end{proof}

\begin{proposition}[Perturbation estimate for the semigroup]
\label{prop:difference}
Let $\alpha$, $\gamma$ and $\eps$ be positive constants.
There exists a constant $C$ such that,
if $0<R\le 1$ and $\delta:=R^{-1}\|\rho-R\|_{C^2}\le \frac1{16}$,
then the semigroup generated
by $A_\rho$ on $H^{2,1}$ satisfies
$$
\| e^{t A_R} - e^{t A_{\rho}} \|_{2,1} \leq C\delta
(1+\log t^{-1})
$$
for all $t\ge 0$. 
\end{proposition}

\begin{proof}
Let $\Gamma$ be the contour consisting
of the two half-lines $\Re\,\lambda= -\frac12\sigma \eps^{-\frac12}
|\Im\, \lambda|$, traversed counterclockwise.
By Lemma~\ref{lem:sectorial}, $\Gamma$ encloses the
spectrum of $A_\rho$ and $A_R$. Since $A_\rho$ 
is sectorial, the
semigroup $e^{tA_\rho}$ is represented by the contour integral
$$
e^{tA_\rho} = \frac{1}{2\pi i}
\oint_{\Gamma} e^{\lambda t} (\lambda-A_\rho)^{-1}\,d\lambda\,,
$$
and correspondingly for $A_R$.
Parametrizing $\Gamma$ by
$\lambda(s) = -|s|+ 2i \sigma^{-1}\sqrt\eps$,
we see that for each $t>0$, the integral converges
absolutely with respect to the operator norm on $H^{2,1}$.

We estimate the difference from $e^{tA_R}$ by
\begin{align*}
\|e^{tA_\rho}-e^{tA_R}\|_{2,1}
&= \left\|\frac{1}{2\pi i}
\oint_{\Gamma} e^{\lambda t} 
\bigl( (\lambda-A_\rho)^{-1}- (\lambda-A_R)^{-1}\bigr)\,d\lambda\right\|_{2,1}\\
&\le \frac{1}{2\pi}
\int_{-\infty}^\infty e^{t\Re \,\lambda(s)}
\|(\lambda\!-\!A_\rho)^{-1} -(\lambda\!-\!A_R)^{-1}\|_{2,1}\, 
|\lambda'(s)|\, ds\\
&\lesssim \delta\int_0^\infty e^{-ts} \min\{1,s^{-1}\}\, ds\,,
\end{align*}
where we have applied Lemma~\ref{lem:diff-res}
to the integrand in the last step. 
For $t\ge 1$, the integral is uniformly bounded.
For $t< 1$, we have
$$
\int_0^\infty e^{-ts} \min\{1,s^{-1}\}\, ds
\le 1 + \int_1^{t^{-1}} \!\!\!s^{-1}\, ds +\int_{t^{-1}}^\infty 
\!t e^{-ts}\, ds
\le 2+\log t^{-1}\,,
$$
proving the claim.
\end{proof}

Next, we address the dependence of the nonlinear evolution
on $\rho$.

\subsection{Proof of Proposition~\ref{prop:continuous-rho}}
Given a solution $u(t)$ of the FHN system on $\cS_R$,
set $\eta = 2 \sup_{0\le t\le T} \|u(t)\|_{2,1} $.
Let $u_\rho(t)$ be the solution on $\cS_\rho$ with
the same initial value, $u_0$.
By definition of the mild solutions,
$$
u_\rho(t)-u(t) = (e^{tA_\rho}-e^{tA_R})u_0
+ \int_0^t \left(
e^{(t-s)A_\rho}N(u_\rho(s)) - e^{(t-s)A_R}N(u(s))\right)\, ds\,,
$$
so long as both solutions exist.
By Proposition~\ref{prop:difference}, for
$\delta\le\frac{1}{16}$ the difference of the semigroups
is bounded by
$$
\|e^{tA_\rho}-e^{tA_R}\|\le C_0\delta(1+\log t^{-1})\,,\qquad (t>0)
$$
with some constant $C_0$.
We use the triangle inequality on the integrand,
and then apply the semigroup estimate and Eq.~\eqref{eq:Lip-N},
\begin{align*} 
\|e^{(t-s)A_\rho}N(u_\rho) -e^{(t-s)A_R}N(u)\|_{2,1}
\hskip -2cm\\
&\le \|(e^{(t-s)A_\rho}-e^{(t-s)A_R})N(u)\|_{2,1}+ 
\|e^{(t-s)A_\rho}(N(u_\rho)-N(u))\|_{2,1} \\
&\le C_0\delta (1+\log t^{-1})\|N(u)\|_{2,1} +
C_1\|N(u_\rho)-N(u)\|_{2,1}
\\
&\le C_0C_\eta \delta (1+\log t^{-1})\|u\|_{2,1} +C_1C_\eta\|u_\rho-u\|_{2,1}
\,,
\end{align*}
so long as $\|u_\rho(s)\|\le \eta$.
Here, $C_1=\sup_t e^{tA_\rho}$, see Lemma~\ref{lem:sectorial},
and $C_\eta$ is the Lipschitz constant Eq.~\eqref{eq:Lip-N}.
For the integral, it follows that
$$
\|u_\rho(t)-u(t)\|_{2,1}
\le C_2(T) \delta \|u_0\|_{2,1} + C_3\int_0^t \|u_\rho(s)-u(s)\|_{2,1}\,ds\,,
$$
where $C_2(t)=C_1(1+2 C_\eta t)(1+\log t^{-1})$, and 
$C_3=C_2C_\eta$.
In the bound on the nonlinearity, we have used
that $\|u_\rho(t)\|_{2,1}\le \frac12 \eta$ for $0\le t\le T$.
Set $C:=C_2(T)e^{C_3 T}$.
By Gr\"onwall's inequality,
$$
\|u_\rho(t)-u(t)\|_{2,1} \le C\delta\|u_0\|_{2,1} \,,
\qquad (0\le t\le T)
$$
provided that $\sup_{0\le t\le T}\|u_\rho(t)\|_{2,1}\le \eta$. Since 
$\|u(t)\|_{2,1}\le \frac12 \eta$ for $0\le t\le T$,
by the triangle inequality 
this is guaranteed by setting
$\delta_*= \min\left\{\tfrac{1}{16}, \tfrac{\eta}{2C}\right\}$.
\hfill $\Box$

We are now ready to construct the near-pulse solutions
on warped cylinders.

\subsection{Proof of Theorem~\ref{thm:persist}}
\label{subsec:proof-persist}
For reference, consider the FHNcyl system on a standard
cylinder $\cS_R$ in a neighborhood of $\cM$.
Fix a pulse $\Phi\in \cM$.
Under the assumptions of
Theorem~\ref{thm:nonli-stab} there are constants $C_0\ge 1$ 
and $\xi_0>0$ and a neighborhood
$\cU$ of $\Phi$ in $H^{2,1}$ such that 
$\dist(u(t),\cM)\le C_0e^{-\xi_0 t}
\|u_0-\Phi\|_{2,1}$ for every solution with initial values in $\cU$.
By translation invariance, it follows that
\begin{equation}
\label{eq:dist-M}
\dist(u(t),\cM)\le C_0 e^{-\xi_0 t}\dist(u_0,\cM)
\end{equation}
for all solutions with initial values in some
tubular neighborhood $\cW$ of $\cM$. We take $\cW$ to have the form
$$
\cW = \left\{w\in H^{2,1}\ \big\vert \ \dist(w,\cM)<\eta\right\}
$$ 
for some $\eta>0$ with $C_0\eta\le\|\Phi\|_{2,1}$. 
Set $T:= \frac{1}{\xi_0}\log(2C_0)$, so that $C_0e^{-\xi_0 T}=\frac12$.
By the triangle inequality,
$$
\sup_{0\le t\le T} \|u(t)\|
\le \sup_{0\le t\le T} \dist(u(t),\cM)+
\sup_{\Phi_h\in \cM} \|\Phi_h\|_{2,1} 
\le 2 \|\Phi\|_{2,1}
$$
for all solutions on $\cS_R$ with initial
value $u\big\vert_{t=0}\in\cW$. 

Consider now the FHNcyl system on a warped cylinder $\cS_\rho$
with $\delta:=R^{-1}\|\rho-R\|_{C^2}\le \delta_*$ 
(to be determined below).
Given an initial condition $u_0\in\cW$, let $u(t)$ be 
the mild solution of the reference system on $\cS_R$
with $u\big\vert_{t=0}=u_0$.
Since $\|u(t)\|_{2,1}\le 2\|\Phi\|_{2,1}$ for all $t\ge 0$, 
by Proposition~\ref{prop:continuous-rho}
there is a value $\delta_0>0$ (determined by $T$) and $C>0$
such that  
\begin{equation}
\label{eq:delta}
\sup_{0\le t\le T}\|u_\rho(t)-u(t)\|_{2,1}\le C\delta \,,
\end{equation} 
provided that $\delta\le \delta_0$.  Choose 
$\delta_*:=\min\{\delta_0,\frac{\eta}{2C}\}$.
Assuming that $\delta\le \delta_*$, we combine
Eq.~\eqref{eq:dist-M} with Eq.~\eqref{eq:delta}
to see that 
\begin{align}
\notag
\dist(u_\rho(t),\cM) &\le \dist(u(t),\cM) +
\|u_\rho(t)-u(t)\|_{2,1}\\
\label{eq:proof-persist-interval}
& \le C_0e^{-\xi_0 t}\dist(u_0,\cM) + C\delta\\
\notag
&< 2C_0\eta
\end{align}
for all $t\in [0,T]$.  It follows that
$\|u_\rho(t)\|_{2,1}\le 2\|\Phi\|_{2,1}$.
Furthermore, by the choice of $\delta_*$ and $T$,
$$
\dist(u_\rho(T),\cM) \le \frac12 \dist(u_0,\cM) + C\delta< \eta\,.
$$
Therefore $u_\rho(T)\in \cW$ whenever
$u_\rho(0)\in \cW$. We repeat the
estimate to obtain inductively
$$
\dist(u_\rho((k\!+\!1)T),\cM) \le 
\frac12 \dist(u_\rho(kT),\cM) + C\delta\,,\qquad (k\in \N)\,.
$$
Solving the recursion, we conclude that
$ \dist(u_\rho(kT),\cM) \le 
2^{-k} \dist(u_0,\cM) + 2 C\delta$,
and further,  by Eq.~\eqref{eq:proof-persist-interval}
$$
\dist(u_\rho(t),\cM) 
\le 
2^{-k} C_0\, \dist(u_0,\cM) + (2+C_0)C\delta
$$
for all $t$ with $kT\le t \le (k+1)T$ and all $k\in\N$.
By the choice of $T$, this implies Eq.~\eqref{eq:persist}
with $\xi= \frac{\log 2}{\log 2 + \log C_0}\xi_0$,
$C_1=2C_0$, 
and $C_2=(2+C_0)C$.
\hfill $\Box$


\end{document}